\documentclass[12pt,a4paper]{article}
\usepackage{amsmath}
\usepackage{amssymb}
\usepackage{amsthm}
\usepackage{amsfonts}
\usepackage{latexsym}
\usepackage{verbatim} 
\usepackage{xcolor}
\usepackage[flushmargin]{footmisc}
\usepackage{color}                    % For creating coloured text and background
\usepackage{mathrsfs}
\usepackage[colorlinks=true,linkcolor=blue,citecolor=blue,urlcolor=blue]{hyperref}

\theoremstyle{plain}

\newtheorem{theorem}{Theorem}[section]
\newtheorem{lemma}[theorem]{Lemma}
\newtheorem{proposition}[theorem]{Proposition}
\newtheorem{corollary}[theorem]{Corollary}
\theoremstyle{definition}
\newtheorem{remark}[theorem]{Remark}
\newtheorem{example}[theorem]{Example}
\theoremstyle{definition}

\newtheorem{assumption}[theorem]{Assumption}
\newtheorem{convention}[theorem]{Convention}

\DeclareMathOperator{\Ker}{Ker}

\DeclareMathOperator{\rad}{rad}
\DeclareMathOperator{\Spec}{Spec}
\DeclareMathOperator{\im}{im}
\DeclareMathOperator{\Sp}{Sp}
\DeclareMathOperator{\Ort}{O}
\DeclareMathOperator{\Ad}{Ad}
\DeclareMathOperator{\ad}{ad}
\DeclareMathOperator{\End}{End}

\DeclareMathOperator{\sgn}{sgn}
\newcommand{\Id}{\mathrm{Id}}
\newcommand{\R}{\mathbb{R}}
\newcommand{\C}{\mathbb{C}}
\newcommand{\g}{\mathfrak{g}}
\newcommand{\lieg}{\mathfrak{g}}
\newcommand{\liegmu}{\mathfrak{g}_\mu}

\newcommand{\lieh}{\mathfrak{h}}
\newcommand{\lies}{\mathfrak{s}}
\newcommand{\ka}{\kappa}
\newcommand{\Z}{\mathbb{Z}}
\newcommand{\Ind}{\operatorname{Ind}}
\newcommand{\id}{\operatorname{id}}

\newcommand{\esym}{\mathfrak{e}}
\newcommand{\fsym}{\mathfrak{f}}
\newcommand{\tsym}{\mathfrak{t}}
\newcommand{\vsym}{\mathfrak{v}}
\newcommand{\gtG}{g_{\tilde G}}
\newcommand{\liek}{\mathfrak{k}}
\newcommand{\OtG}{\Omega_{\tilde G}}
\newcommand{\orb}{\mathcal{O}_\mu}

\newcommand{\TG}{T^*G}
\newcommand{\gad}{g_{\mathrm{ad}}}
\newcommand{\Jad}{J_{\mathrm{ad}}}
\newcommand{\gred}{g_{\mathrm{red}}}
\newcommand{\pair}[2]{\langle #1,\, #2\rangle}
\newcommand{\kah}{\ka_{\lieh}}
\newcommand{\Gform}{\mathcal{G}}

\title{Pseudo--Kähler induction on Lie groups with entire Grauert tubes}
\author{Gabriele Barbieri\footnote{ {\bf e-mail}: gabriele.barbieri@unicatt.it, {\bf Address:} Dipartimento di Scienze Matematiche Fisiche e Naturali, Universit\`a Cattolica del Sacro Cuore, via della Garzetta 48 - Brescia, Italy, {\bf ORCID ID:} 0009-0000-9539-7693 }
\text{ and} Andrea Galasso\footnote{{\bf e-mail}: andrea.galasso@unipegaso.it, andrea.galasso.91@gmail.com, {\bf Address:} Dipartimento di Ingegneria, Universit\`a Pegaso, Sede Centrale Centro Direzionale Isola F2 - Napoli, Italy,
{\bf ORCID ID:} 0000-0002-5792-1674 }}

\date{}

\begin{document}
\maketitle

\begin{abstract} Given a closed subgroup $H\subseteq G$ and a Hamiltonian $H$-space $Y$, one can construct an induced Hamiltonian $G$-space. In this paper we investigate this construction in the setting where $Y$ is endowed with a compatible complex structure. Our aim is to develop symplectic induction in this framework and to establish the corresponding K\"ahler analogues of the classical results. The main idea is to replace the cotangent bundle $T^*G$, appearing in the standard construction of $\operatorname{Ind}$, with the Grauert tube of $G$. We focus on Lie groups admitting a globally defined Grauert tube and study the resulting pseudo--K\"ahler induction procedure. 
\end{abstract}

\newpage

\tableofcontents


\newpage

\section{Motivation and introduction}

%Let $G$ be a connected Lie group and $H \subseteq G$ be an even-codimensional closed subgroup, then a homogeneous almost K\"ahler structure on the coset space $G/H$ is completely determined by a suitable linear structure on the Lie algebra $\mathfrak{g}$. In that setting, one can apply geometric quantization to $G/H$ and study certain unitary representations that arise when one quantizes the action. For compact Lie groups this construction furnishes a one-to-one correspondence between representations and coadjoint orbits, via the Borel-Weil theorem. This correspondence has provided strong motivation for the study of the interplay between representation theory and quantization in symplectic geometry.

Symplectic analogues of the induced representation $\mathrm{Ind}$ and space of intertwiners $\mathrm{Hom}$ were introduced in \cite{Kazhdan:1978, Weinstein:1978, Guillemin:1982}, motivated by the correspondence between Hamiltonian spaces and unitary $G$-modules (exemplified by e.g. the Borel-Weil theorem for compact groups \cite{Serre:1954} and Kirillov-Bernat theory for exponential groups \cite{Kirillov:1962}). These constructions were extended from coadjoint orbits to Hamiltonian and prequantum spaces in \cite{RatiuZiegler}, where symplectic and prequantum versions of the Frobenius reciprocity and induction-in-stages property were proved.

In geometric quantization, see \cite{Kostant:1970} and \cite{Souriau:1970}, we are interested in reducing the space of sections of the quantizing circle bundle using polarization. In this paper, we incorporate a complex structure into the picture. Note that the definition of $\mathrm{Ind}$ involves $T^*G$, and thus we are required to define a complex structure on $T^*G$. For a compact Lie group $G$ with a bi-invariant Riemannian metric, there exists a natural complex structure $J_{\mathrm{ad}}$ which is defined on the whole cotangent bundle $T^*G$, and it induces a biholomorphism from $T^*G$ to the complexification of $G$, see \cite{Szoke}. Then, $T^*G$ has a natural K\"ahler structure and the construction in the previous paragraph can easily be adapted to this setting. Following this strategy, one should replace $T^*G$ with the Grauert tube $\mathcal{T}_G$ of $G$ in the construction of the $\mathrm{Ind}$ functor. Let $\mathfrak{g}$ be a real Lie algebra of Loeb type, i.e.
\[
\mathrm{Spec}(\mathrm{ad}_\xi) \cap \mathbb{R} = \{0\}, \qquad \forall \xi \in \mathfrak{g}.
\]
Let $\mathfrak{g}^\mathbb{C} = \mathfrak{g} \otimes_{\mathbb{R}} \mathbb{C}$ be its complexification. Then $\mathfrak{g}^\mathbb{C} = \mathfrak{g} \oplus i\mathfrak{g}$ as real vector spaces, and it is a complex Lie algebra with the complex bilinear extension of the Lie bracket. If $G$ is a Lie group of Loeb type the Grauert tube is entire, which means that $\mathcal{T}_G=T^*G$ and $J_{\mathrm{ad}}$ is compatible with the canonical symplectic two form on $T^*G$. The resulting structure on $T^*G$ is pseudo-K\"ahler. The first aim of this paper is to study this latter case.

More precisely, given a closed subgroup $H\subseteq G$
and a Hamiltonian $H$-space $(Y,\omega_Y,\Psi)$, set
\begin{equation}\label{eq:N}
	N:=T^*G\times Y,\qquad \omega_N:=d\varpi+\omega_Y,
\end{equation}
where $\varpi$ is the canonical $1$-form of $T^*G$, and consider the
$G\times H$-action $$(g,h)(p,y)=(gph^{-1},h(y))$$ and its equivariant momentum
map $\phi\times\psi:N\to\lieg^*\times\lieh^*$,
\begin{equation}\label{eq:RZmom}
	\phi(p,y)=pq^{-1},\qquad
	\psi(p,y)=\Psi(y)-q^{-1}p\,\big|_{\lieh}\qquad (p\in T^*_qG)\,.
\end{equation}
The  \emph{induced Hamiltonian $G$-space} is the Marsden--Weinstein quotient at $0\in\lieh^*$,
\begin{equation}\label{eq:Ind}
	\Ind_H^G Y:=N/\!\!/H=\psi^{-1}(0)/H \,.
\end{equation}

Two additional geometric ingredients are fixed. First, an
$\Ad(G)$-invariant and non-degenerate symmetric bilinear form
$\ka=\ka_e$ on $\lieg$ (a pseudo-Euclidean product), inducing a bi-invariant pseudo-Riemannian metric on $G$, and an identification $\lieg^*\cong\lieg$. In the spirit of the adapted complex structures, see \cite{Szoke}, it induces a pseudo-K\"ahler structure $$(\omega_{\TG}=d\varpi,\Jad,\gad)$$ on $T^*G$. Second, a pseudo-K\"ahler structure $(\omega_Y,J_Y)$ on $Y$, with associated pseudo-Riemannian metric $g_Y:=\omega_Y(J_Y\cdot,\cdot)$ (Assumption \ref{ass:Y} below). The product $N$ then carries the pseudo-Riemannian structure
\begin{equation}\label{eq:gN}
	g_N:=\gad\oplus g_Y .
\end{equation}

The pseudo-K\"ahler reduction theorem we wish to invoke is the following (restated from \cite[Appendix B]{rungi}).

\begin{theorem}[{pseudo-K\"ahler reduction, \cite[Thm.~B.6]{rungi}}]
	\label{thm:B6}
	Let a Lie group $H$ act on a pseudo-K\"ahler manifold $(M,\omega,J,g)$ preserving $\omega$, $J$ and $g$, with equivariant momentum map $\Phi_M:M\to \mathfrak{h}^*$. Assume the induced action on
	$\Phi^{-1}_M(0)$ is free and proper, and that
	\begin{equation}\label{eq:B6hyp}
		g\big|_{T_p(H\cdot p)}\ \text{is non-degenerate for every }
		p\in\Phi_M^{-1}(0).
	\end{equation}
	Then $\Phi^{-1}_M(0)/H$ is a smooth manifold carrying a unique	pseudo-Riemannian metric $\gred$ and complex structure $J_{\mathrm{red}}$ whose pull-backs agree with the restrictions of $g$ and $J$, and
	$\omega_{\mathrm{red}}=\gred(\cdot,J_{\mathrm{red}}\cdot)$ is symplectic.
\end{theorem}

Our aim is to apply Theorem \ref{thm:B6} to $N$. Thus, we need to study when hypothesis \eqref{eq:B6hyp} holds for the $H$-action on $\psi^{-1}(0)\subset N$. This is a genuine issue already when $Y$ is a point, where $N/\!\!/H=T^*(G/H)$, see Example \ref{ex:TGH} below. 

We also remark that the constructions of this paper admit natural contact and CR counterparts. Replacing a quantizable pseudo-K\"ahler manifold by its quantizing circle bundle, symplectic reduction is replaced by contact and CR reduction, and geometric quantization by the study of the Szeg\H{o} projector; in this setting, theorems of ``quantization commutes with reduction'' type are available, see \cite{hmm}. For a survey on the topic, we refer to \cite{ma}, where the connection with the CR setting is also highlighted. For a strict quantization of compact pseudo-K\"ahler manifolds, compatible with group actions, contextually relevant to this article we refer to \cite{Galasso}. Since the objects considered here lift naturally to the corresponding circle bundles, we expect our constructions to transport to that category with only minor changes; we plan to return to this elsewhere.

Eventually, we explain the organization of this paper. In Section \ref{sec:oscillatorexample} we work out the example of the oscillator Lie groups, which motivates our study and is inspired by \cite{Streater}: we describe the coadjoint orbits and the flat K\"ahler structure of the paraboloid orbit. %For a real \emph{semisimple} group $G$, invariant pseudo-K\"ahler structures on adjoint orbits are tied to ellipticity of the base point: the fine structure of reductive pseudo-K\"ahlerian homogeneous spaces was determined by Dorfmeister--Guan \cite{DorfmeisterGuanGD}, and elliptic adjoint orbits are known to carry invariant pseudo-K\"ahler structures (see e.g.\\cite{Boumuki}); a canonical compatible almost complex structure on adjoint orbits of non-compact semisimple groups, defined through root data, is  studied in \cite{DellaVedovaGatti}. The construction presented below is elementary, is phrased through the functional calculus of Lemma~\ref{lem:calc} rather than through root data, and applies to arbitrary Lie algebras endowed with an invariant pseudo-Euclidean product satisfying \textup{(K1)}--\textup{(K2)} --- in particular to \emph{solvable} {\color{red}groups such} as the oscillator groups of Section~\ref{sec:oscillatorexample}, which lie outside the semisimple theory.
In Section \ref{sec:grauert} we review the theory of Grauert tubes for Loeb type Lie groups. In Section \ref{sec:complexorbit} we explain how to define a natural complex structure on some co-adjoint orbits; our method resembles the one used for compact Lie groups. We recall that, in the setting of solvable Lie groups, a theory of polarizations was developed in \cite{AuslanderKostant:1971}. The two settings are transversal --- Loeb type groups include the compact ones, while not every solvable group is of Loeb type. They overlap, however: nilpotent groups, and the oscillator group of Section~\ref{sec:oscillatorexample}, are both solvable and of Loeb type. It would therefore be interesting to compare the two approaches. In Section \ref{sec:psN} we apply the pseudo-K\"ahler reduction of \cite[Theorem B.6]{rungi} to our setting, and we show by examples --- notably $T^*(G/H)$, see Example \ref{ex:TGH} --- that it is not implied by purely algebraic conditions. In Section \ref{sec:application} we apply our construction to coadjoint orbits, completing the example of the oscillator group, and we prove that self-induction reproduces the canonical structure of the orbit (Theorem \ref{thm:reproduce}). Finally, in Section \ref{sec:stages} we prove a pseudo-K\"ahler version of the induction-in-stages property, condition \eqref{eq:B6hyp} has to be imposed at each stage and we show a counterexample in Example \ref{ex:stagesdiff}.

\section{An example}\label{sec:oscillatorexample}

We consider oscillator Lie groups; we give an example based on \cite{Szoke}, \cite{Streater} and \cite{as}. Let $$(\lambda_1,\dots,\lambda_n)\in (\mathbb{R}^+)^n\,.$$ Let $\mathfrak g$ be the real Lie algebra with basis
\[
X_1,\dots,X_n,\; Y_1,\dots,Y_n,\; Z,\; W,
\]
where all the brackets are zero except
\[
[X_i,Y_i]=\lambda_i Z,\qquad
[W,X_i]=\lambda_i Y_i,\qquad
[W,Y_i]=-\lambda_i X_i.
\]

Note that the ideal spanned by $X_1,\dots,X_n,Y_1,\dots,Y_n,Z$ is the Heisenberg algebra $\mathfrak h_n$.
Moreover,
\[
\mathfrak g = \mathbb{R}W \ltimes \mathfrak h_n.
\]

Let $V$ be the subspace spanned by $\{X_1,\dots,X_n,Y_1,\dots,Y_n\}$, and $S$ the subspace spanned by $\{Z,W\}$. Define a symmetric bilinear form $\kappa_e$ on $\mathfrak g$ by taking
\[
\{X_1,\dots,X_n,Y_1,\dots,Y_n\}
\]
to be an orthonormal set, $W$ and $Z$ to be isotropic with
\[
\kappa_e(Z,W)=1,
\]
and the subspaces $V$ and $S$ to be orthogonal. This form has signature $(2n+1,1)$. By \cite[Section 3]{as} $\kappa_e$ is $\operatorname{ad}(\mathfrak g)$-invariant, i.e., for all $\xi,\zeta,\eta\in\mathfrak g$,
\[
\kappa_e([\xi,\zeta],\eta) + \kappa_e(\zeta,[\xi,\eta]) = 0.
\]

Now, we show explicit computation for the case $n=1$. Consider the four matrices (the definition of \(H\) in \cite{Streater} should be modified by a change of sign):
\[
H=
\begin{pmatrix}
	0&0&0&0\\
	0&0&-1&0\\
	0&1&0&0\\
	0&0&0&0
\end{pmatrix},
\qquad
P=
\begin{pmatrix}
	0&0&1&0\\
	0&0&0&1\\
	0&0&0&0\\
	0&0&0&0
\end{pmatrix},
\]
\[
Q=
\begin{pmatrix}
	0&-1&0&0\\
	0&0&0&0\\
	0&0&0&1\\
	0&0&0&0
\end{pmatrix},
\qquad
E=
\begin{pmatrix}
	0&0&0&2\\
	0&0&0&0\\
	0&0&0&0\\
	0&0&0&0
\end{pmatrix}.
\]

One verifies directly that
\[
[H,P]=Q,\qquad
[H,Q]=-P,\qquad
[P,Q]=E.
\]
Now, we compute a generic element of the group. Since $P^2=Q^2=E^2=0$, their exponentials reduce to
\[
e^{xP}=I+xP,\qquad
e^{yQ}=I+yQ,\qquad
e^{\alpha E}=I+\alpha E.
\]
Furthermore,
\[
e^{tH}
=
\begin{pmatrix}
	1&0&0&0\\
	0&\cos t&-\sin t&0\\
	0&\sin t&\cos t&0\\
	0&0&0&1
\end{pmatrix}.
\]
A direct computation gives	$PQ= E/2$, and hence
\[
(I+xP)(I+yQ)
=
I+xP+yQ+\frac{xy}{2}E,
\]
so that
\[
e^{\alpha E}e^{xP}e^{yQ}
=
I+xP+yQ+\left(\alpha+\frac{xy}{2}\right)E.
\]
Multiplying by $e^{tH}$ on the right, a generic element $g$ is given by
\[	e^{\alpha E}e^{xP}e^{yQ}e^{tH}
=
\begin{pmatrix}
	1 &
	-y\cos t+x\sin t &
	y\sin t+x\cos t &
	2\alpha+xy
	\\[2mm]
	0 & \cos t & -\sin t & x
	\\
	0 & \sin t & \cos t & y
	\\
	0 & 0 & 0 & 1
\end{pmatrix}.
\]

At this point, let $\{H',P',Q',E'\}$ denote the dual basis of $\mathfrak{g}^*$, so that $\langle H',H\rangle = 1$, etc.	For $\mu = h_H H' + h_P P' + h_Q Q' + h_E E' \in \mathfrak{g}^*$ and $X = aH+bP+cQ+dE \in \mathfrak{g}$, the coadjoint action is
\[
\mathrm{ad}^*_X(\mu)(Y) = -\mu([X,Y]).
\]
A direct computation (using the brackets above, and noting that $d$ never contributes since $E$ is central) gives
\[
\mathrm{ad}^*_X(\mu) = (b\,h_Q - c\,h_P)\,H' + (c\,h_E - a\,h_Q)\,P' + (a\,h_P - b\,h_E)\,Q'.
\]
In particular the $E'$-component of every tangent vector vanishes.

The orbit through $\mu_0 = H'+E'$ (i.e.\ $h_H=1$, $h_E=1$, $h_P=h_Q=0$; note that $H'$ itself is a fixed point of the coadjoint action, since every tangent vector has vanishing $E'$-component and vanishes when $h_P=h_Q=h_E=0$) under the connected group generated by $\exp(\alpha E)\exp(xP)\exp(yQ)\exp(tH)$ is the paraboloid
\[
\mathcal{O}_{H'+E'} = \left\{\, h_H H' + h_P P' + h_Q Q' + h_E E' : h_E = 1,\, h_H = 1 + \tfrac{1}{2}\big(h_P^2+h_Q^2\big) \,\right\},
\]
parametrized by $(h_P,h_Q)\in\mathbb{R}^2$.

Now, we write explicit expression for the tangent vectors to the orbit. Writing $w_X := \mathrm{ad}^*_X(\mu)$ for $X\in\{H,P,Q\}$ at a point $\mu$ with coordinates $(h_H,h_P,h_Q,h_E)$:
\[
w_H = -h_Q\,P' + h_P\,Q', \qquad
w_P = h_Q\,H' - h_E\,Q', \qquad
w_Q = -h_P\,H' + h_E\,P'.
\]
These satisfy the linear relation
\[
h_E\, w_H + h_P\, w_P + h_Q\, w_Q = 0,
\]
so for $h_E\neq 0$ the tangent space is
\[
T_\mu\mathcal{O}_{H'+E'} = \mathrm{span}\{w_P, w_Q\}, \qquad \dim T_\mu\mathcal{O}_{H'+E'} = 2.
\]

Natural coordinates on $\mathcal{O}_{H'+E'}$ are $(h_P,h_Q)$, with coordinate vector fields
\[
\partial_{h_P} = \frac{1}{h_E}\,w_Q, \qquad \partial_{h_Q} = -\frac{1}{h_E}\,w_P.
\]

On $\mathfrak{g}^*$, let $\kappa^*$ be the symmetric bilinear form induced by $\kappa_e$, namely $\kappa^*(\alpha,\beta):=\kappa_e(\alpha^\sharp,\beta^\sharp)$, where $\sharp$ inverts $\xi\mapsto\kappa_e(\xi,\cdot)$; explicitly $H'^{\,\sharp}=E$, $E'^{\,\sharp}=H$, $P'^{\,\sharp}=P$, $Q'^{\,\sharp}=Q$, so that
\[
\kappa^*(P',P') = \kappa^*(Q',Q') = 1, \qquad \kappa^*(H',E') = \kappa^*(E',H') = 1,
\]
all other pairings vanishing. In matrix form,
\[
\kappa^* = \begin{pmatrix} 0 & 0 & 0 & 1 \\ 0 & 1 & 0 & 0 \\ 0 & 0 & 1 & 0 \\ 1 & 0 & 0 & 0 \end{pmatrix}, \qquad \det \kappa^* = -1,
\]
with eigenvalues $\{1,1,1,-1\}$: $\kappa^*$ is non-degenerate of signature $(3,1)$, but not positive definite. The pair $\{H',E'\}$ spans a hyperbolic plane, while $\{P',Q'\}$ spans an orthogonal Euclidean plane.

Now, we write the restriction of $\kappa^*$ to $T_\mu\mathcal{O}_{H'+E'}$. For $v=\mathrm{ad}^*_X(\mu)$ and $v'=\mathrm{ad}^*_{X'}(\mu)$ (with $X=aH+bP+cQ+dE$, $X'=a'H+b'P+c'Q+d'E$), write
\[
v = v_H H' + v_P P' + v_Q Q', \qquad v' = v'_H H' + v'_P P' + v'_Q Q'.
\]
Since every tangent vector has vanishing $E'$-component, and $\kappa^*(H',\cdot)$ pairs only with $E'$, the $H'$-component of $v$ and $v'$ never contributes to $\kappa^*(v,v')$. Explicitly,
\[
\kappa^*(v,v') = v_P v'_P + v_Q v'_Q = (c h_E - a h_Q)(c' h_E - a' h_Q) + (a h_P - b h_E)(a' h_P - b' h_E).
\]

In the basis $\{w_P,w_Q\}$ the Gram matrix is
\[
\kappa^*\big|_{T_\mu\mathcal{O}_{H'+E'}} = \begin{pmatrix} h_E^2 & 0 \\ 0 & h_E^2 \end{pmatrix} = h_E^2\, I_2, \qquad \det = h_E^4 \neq 0 \ \ (h_E\neq 0).
\]
In the coordinate basis $\{\partial_{h_P},\partial_{h_Q}\}$ this simplifies to the identity:
\[
\kappa^*\big|_{\mathcal{O}_{H'+E'}} = dh_P^{\,2} + dh_Q^{\,2}.
\]
Thus, $\kappa^*$ restricts to a non-degenerate, positive-definite (Riemannian), and \emph{flat} metric on $\mathcal{O}_{H'+E'}$ at every point.

Eventually, the {Kirillov--Kostant--Souriau} symplectic form on a coadjoint orbit is defined by
\[
\omega_\mu\big(\mathrm{ad}^*_X(\mu),\, \mathrm{ad}^*_Y(\mu)\big) = -\langle \mu, [X,Y]\rangle\,,
\]
consistently with the convention \eqref{eq:kks} adopted in Section \ref{sec:orbits} below (under the identification $\g^*\cong\g$).
Using $[X,Y] = (a'c-ac')\,P + (ab'-a'b)\,Q + (bc'-b'c)\,E$, we obtain
\[
\omega_\mu\big(\mathrm{ad}^*_X(\mu),\, \mathrm{ad}^*_Y(\mu)\big) = (ac'-a'c)\,h_P - (ab'-a'b)\,h_Q - (bc'-b'c)\,h_E.
\]

On the generators:
\[
\omega(w_P,w_Q) = -h_E, \qquad \omega(w_H,w_P) = -h_Q, \qquad \omega(w_H,w_Q) = h_P.
\]
In intrinsic form, for tangent vectors with components $(v_P,v_Q)$ and $(v'_P,v'_Q)$,
\[
\omega_\mu(v,v') = \frac{v_Q v'_P - v_P v'_Q}{h_E},
\]
and in the coordinates $(h_P,h_Q)$,
\[
\omega = -\frac{1}{h_E}\, dh_P \wedge dh_Q.
\]

Thus $(\mathcal{O}_{H'+E'},\omega)$ is symplectomorphic to $(\mathbb{R}^2, -\tfrac{1}{h_E}\,dp\wedge dq)$, i.e.\ the standard symplectic plane up to a constant rescaling.

The two structures are related through the complex structure
\[
J(v_P,v_Q) = (-v_Q,\,v_P)\,,
\qquad
\gamma(v,v'):=\omega(Jv,v') = \frac{1}{h_E}\,\kappa^*(v,v')\,,
\]
where $J$ is the $90^\circ$ rotation in the $(h_P,h_Q)$-plane --- precisely the infinitesimal rotation generated by $H$, whose flow is $$(h_P,h_Q)\mapsto(\cos t\,h_P-\sin t\,h_Q,\ \sin t\,h_P+\cos t\,h_Q),$$ as one reads off from $w_H$. Thus the triple
\[
\Bigl(\omega\,,\ J\,,\ \gamma=\tfrac1{h_E}\bigl(dh_P^{\,2}+dh_Q^{\,2}\bigr)\Bigr)
\]
equips $\mathcal{O}_{H'+E'}$ with the structure of a \emph{flat K\"ahler manifold}: this is the geometric realization, via the orbit method, of the phase space of the harmonic oscillator. As we shall see, $(J,\gamma)$ is exactly the \emph{canonical} invariant pseudo-K\"ahler structure produced by Theorem \ref{thm:orbit} of Section \ref{sec:orbits}; the corresponding induction picture is completed in Example \ref{ex:oscillator} of Section \ref{sec:application}.

\section{Entire Grauert tubes}
\label{sec:grauert}

The literature on Grauert tubes is extensive, and we do not attempt a comprehensive review here. Instead, we refer the reader to \cite{Szoke}, which is the most relevant to our present work. Let $M$ be an $n$--dimensional complete (not necessarily compact) connected real--analytic manifold. The Bruhat--Whitney complexification of $M$ is a complex manifold $(\tilde M,J_{\tilde M})$ in which $M$ embeds as a totally real submanifold. The choice of a real--analytic pseudo--Riemannian metric $\kappa_M$ on $M$ determines a canonical pseudo--Kähler structure on {domains of} $(\tilde M,J_{\tilde M})$, symplectomorphic to tubular neighborhoods of the zero section in the tangent bundle of $M$, which we now define.

For $\tau>0$, let $T^\tau M$ be the tubular neighborhood of radius $\tau$ of the zero section in $TM$. For sufficiently small $\tau$ there exists an intrinsic so--called \emph{adapted complex structure} $J_{\mathrm{ad}}$ on $T^\tau M$. Furthermore $J_{\mathrm{ad}}$ is compatible with the canonical symplectic structure $\Omega_{\mathrm{can}}$ on $TM$, and the associated pseudo--Riemannian metric restricts to $\kappa_M$ along $M$. For a general $(M,\kappa_M)$ the structure $J_{\mathrm{ad}}$ need not be defined on the whole tangent bundle $TM$.

An equivalent viewpoint, see \cite[Section~8]{Szoke},
is obtained by starting from the canonical symplectic
structure $\omega_{T^*M}$ on the cotangent bundle $T^*M$.
Instead of first constructing the complexification and then
endowing it with the pseudo--Kähler structure induced by
$\kappa_M$, one introduces a compatible complex structure
$J_{M,\mathrm{ad}}$ on a suitable neighborhood of the zero section in $(T^*M,\omega_{T^*M})$. Such a neighborhood is called a \emph{Grauert tube}. The complex structure $J_{M,\mathrm{ad}}$ is canonically
determined by the metric $\kappa_M$ and is called the
\emph{adapted complex structure}. The Grauert tube is said to be \emph{entire} if the adapted complex structure exists on the whole cotangent bundle $T^*M$.

If $G$ is compact and semisimple, then the metric $\kappa_G$
is Riemannian and the associated Grauert tube is entire. Examples of manifolds carrying entire pseudo--Riemannian metrics are not easy to construct; see \cite{Szoke}. On the other hand, in \cite{Medina} it was shown that the so--called oscillator Lie algebras admit invariant Lorentzian bilinear forms. The corresponding simply connected Lie groups therefore admit a bi--invariant Lorentzian metric and they are the only solvable non--commutative simply connected Lie groups
with this property. The Grauert tubes of these Lie groups are entire, see \cite[Theorem~0.5]{Szoke}. In the next subsection, we introduce notation and we make the preceding discussion more precise. 

%We also refer to \cite{HalverscheidIannuzzi} for the study of the Grauert tube on $SL(2,\mathbb R)$. However, by Cartan's criteria for solvability and semi-simplicity, these structures cannot be induced by the Killing form.

\subsection{Lie groups with entire Grauert tubes}

In this subsection we recall standard preliminaries about Grauert tubes of Lie groups, the main references for this section are \cite{Szoke} and \cite{paogal}. Note that in \cite{paogal} they consider the case of compact Lie group, and we refer to \cite{paogal} for more background. Here, we adopt the perspective of \cite{Szoke} and we always assume $G$ being a connected Lie group of Loeb type. Denote by $\mathfrak g$ the Lie algebra of left--invariant vector fields, whose dual is $\mathfrak g^{*}$. We shall canonically identify $\mathfrak g$ with the tangent space of $G$ at the identity $e \in G$, namely $T_eG$. The dimension of $G$ will be denoted by $d_G$.

Let $\kappa_e$ be an $\mathrm{Ad}$--invariant pseudo--Euclidean product on $\mathfrak g$, where $\mathrm{Ad}$ denotes the adjoint representation. We denote by $\kappa^{*}$ the corresponding product on $\mathfrak g^{*}$ and by $\kappa$ the induced pseudo--Riemannian metric on $G$. The complexification $(\tilde G,J_{\tilde G})$ of $G$ is a connected complex $d_G$--dimensional Lie group endowed with a complex structure $J_{\tilde G}$. Its Lie algebra $\tilde{\mathfrak g}$ is the complexification of $\mathfrak g$, $\tilde{\mathfrak g} = \mathfrak g \otimes \mathbb C$.

Let $\exp_{\tilde G} : \tilde{\mathfrak g} \to \tilde G$ be the exponential map. Then
\[
P(g,\xi) := g\,\exp_{\tilde G}(i\xi),
\qquad g\in G,\ \xi\in\mathfrak g \,.
\]
following \cite{Szoke}, $P$ is called the {polar map}. Since $G$ is of Loeb type, by \cite[Proposition 7.4]{Szoke} the map $P : G \times \mathfrak g \longrightarrow \tilde G$ is a local diffeomorphism at each point of $G\times\mathfrak g$, and thus it is everywhere regular; an explicit formula for $dP$, together with a self-contained proof of this regularity, is given in Lemma \ref{lem:dP} below.

In the sequel we shall use the shorthand notation
\[
g\cdot \xi := \mathrm{d}_e L_g(\xi), \qquad
\xi\cdot g := \mathrm{d}_e R_g(\xi),
\]
where $L_g,R_g : G\to G$ denote the left and right translations by
$g\in G$, respectively. Composing $P$ with the trivialization of the tangent bundle $TG$ given by left translations to the identity, we obtain the following.
If $v\in T_gG$, then $v=g\cdot\xi$ for a unique $\xi\in\mathfrak g$. This yields 
\[
E : TG \longrightarrow \tilde G,
\qquad
E(g,v)=E(g,g\cdot\xi)=g\,\exp_{\tilde G}(i\xi).
\]

In the present setting, however, the following result holds.

\begin{theorem}[{\cite[Theorem 7.5]{Szoke}}]\label{thm:Szoke}
	Let $(G,\kappa_G)$ be a Lie group of Loeb type. Then the adapted complex	structure $J_{\mathrm{ad}}$ is defined on $TG$, and the map
	\[
	E : (TG,J_{\mathrm{ad}}) \longrightarrow (\tilde G,J)
	\]
	is a biholomorphism.
\end{theorem}

This result puts the theory of complex reductive groups in contact with the general theory of Grauert tubes. In particular, Theorem \ref{thm:Szoke}, Theorem \cite[Theorem 0.5 and Theorem 0.6]{Szoke} implies that $(TG,J_{\mathrm{ad}},\Omega_{\mathrm{can}})$ is a pseudo-Kähler manifold. Furthermore, if
	\[
	\OtG := (E^{-1})^*(\Omega_{\mathrm{can}}),
	\]
then $(\tilde G,J_{\tilde G},\OtG)$ is also a pseudo-Kähler manifold and $E$ is an isomorphism of pseudo--Kähler manifolds. By \cite[Theorem 0.6]{Szoke} the induced pseudo-Riemannian metric
	\[
	\gtG(\cdot,\cdot):=\OtG(J_{\tilde G}\cdot,\cdot)
	\]
on $\tilde G$ restricts to $\kappa$ along $G$. (We adopt throughout the convention $g=\omega(J\cdot,\cdot)$ for the metric of a compatible pair; the opposite convention $\omega(\cdot,J\cdot)$, also common in the literature, amounts to replacing $J$ by $-J$; cf.\ Convention \ref{conv:g}.)
	
Now, consider the norm function $\|\cdot\|_\kappa$ on $TG$, defined as	the pull--back of the norm $\|\cdot\|_{\kappa_e}$ on $\mathfrak g$	under the above trivialization:
	\[
	\|(g,g\cdot\xi)\|_\kappa
	:=
	\|g\cdot\xi\|_{\kappa_g}
	=
	\|\xi\|_{\kappa_e}.
	\]
	
For $x\in\tilde G$, let
\[	\mathrm{val}_x:\mathfrak g\to T_x\tilde G,	\qquad	\xi\mapsto \xi_{\tilde G}(x),	\]
where $\xi_{\tilde G}\in\mathfrak X(\tilde G)$ is the vector field induced by $\xi$ under the action of $G$ on $\tilde G$ by left	translations. Then
\[ \sigma_x := \mathrm{val}_x^*\bigl((\gtG)_x\bigr)\]
is a pseudo--Euclidean scalar product on $\mathfrak g$. For a vector subspace $\mathfrak h\subset\mathfrak g$ we define the vector subbundle
\[	\mathfrak h_{\tilde G}(x):=\mathrm{val}_x(\mathfrak h)\,.	\]

{\begin{remark}\label{rem:sigma}
		In the notation of Section \ref{sec:polar} below, at $x=P(q,\zeta)$ one computes
		$$\sigma_x(\xi,\eta)=\kappa_e\bigl(\fsym(\ad_\zeta)\Ad_{q^{-1}}\xi,\,\Ad_{q^{-1}}\eta\bigr)
		=\kappa_e\bigl(\fsym(\ad_{\zeta_q})\,\xi,\,\eta\bigr),
		\qquad \zeta_q:=\Ad_q\zeta,$$
		with $\fsym(z)=z\cot z$, since the generator of the left action is the
		horizontal vector $(\Ad_{q^{-1}}\xi,0)$ in the trivialization
		\eqref{eq:ident}, the second equality following from
		$\Ad_q\ad_\zeta\Ad_{q^{-1}}=\ad_{\Ad_q\zeta}$ and the $\Ad$-invariance of
		$\kappa_e$. In particular $\sigma_x$ is non-degenerate for every $x$
		(Lemma \ref{lem:loeb}), and, since $\fsym(0)=1$ and
		$\zeta_q\in\ker\ad_{\zeta_q}$ force $\fsym(\ad_{\zeta_q})\zeta_q=\zeta_q$,
		the radial vector $\zeta_q$ satisfies
		$\zeta_q^{\perp_{\sigma_x}}=\zeta_q^{\perp_{\kappa_e}}$: this is the content
		of \cite[Lemma 4.11]{paogal} in the present setting. For a subspace
		$\lieh$, however, one only has
		$\lieh^{\perp_{\sigma_x}}=\bigl(\fsym(\ad_{\zeta_q})\lieh\bigr)^{\perp_{\kappa_e}}$,
		which differs from $\lieh^{\perp_{\kappa_e}}$ in general.
\end{remark}}

\section{Complex structures on coadjoint orbits}
\label{sec:complexorbit}

\subsection{The linear--algebraic construction}\label{sec:linalg}

Let $V$ be a finite-dimensional real vector space and let $\omega\colon V\times V\to\R$ be an antisymmetric bilinear form, \emph{possibly degenerate}, with radical
\[
W:=\rad\omega=\{\,w\in V:\ \omega(v,w)=0\ \ \forall v\in V\,\}.
\]
Let $g$ be a pseudo-Euclidean product on $V$ (a nondegenerate symmetric
bilinear form of arbitrary signature){, and let $A\in\End(V)$ be defined by}
\begin{equation}\label{eq:defA}
	\omega(v,w)=g(v,Aw)\qquad\forall\,v,w\in V,
\end{equation}
{the existence and uniqueness of such an $A$ being consequences of the non-degeneracy of $g$.}
For $T\in\End(V)$ we write $T^{*}$ for the $g$-adjoint, characterized by
$g(Tv,w)=g(v,T^{*}w)$, and $S^{\perp}$ for the $g$-orthogonal complement of a
subspace $S\subseteq V$. Spectra always refer to the complexified operators on
$V_{\C}=V\otimes_{\R}\C$; bilinear forms are extended $\C$-bilinearly, and
$\sigma\colon V_{\C}\to V_{\C}$ denotes the conjugation fixing $V$. An operator
$T\in\End(V_{\C})$ is called \emph{real} if $\sigma T\sigma=T$, i.e.\ if it
preserves $V$.

\medskip
\noindent\textbf{Hypotheses.} Throughout this section we assume:
\begin{itemize}
	\item[\textup{(H1)}] $W\cap W^{\perp}=\{0\}$;
	\item[\textup{(H2)}] $\Spec(A)\cap\R=\{0\}$.
\end{itemize}

The following two lemmas are proved via standard arguments in linear algebra.

\begin{lemma}\label{lem:kerim}
	The operator $A$ is $g$-antiadjoint, $A^{*}=-A$, and
	\[
	\ker A=W,\qquad \im A=W^{\perp}.
	\]
\end{lemma}

\begin{proof}
	By antisymmetry of $\omega$, by \eqref{eq:defA} and by symmetry of $g$,
	\[
	g(Av,w)=g(w,Av)=\omega(w,v)=-\omega(v,w)=-g(v,Aw)\qquad\forall v,w\in V,
	\]
	so $A^{*}=-A$. By nondegeneracy of $g$,
	\[
	Aw=0\iff g(v,Aw)=0\ \ \forall v\iff\omega(v,w)=0\ \ \forall v\iff w\in W,
	\]
	hence $\ker A=W$. Finally, using nondegeneracy of $g$, $$\im A=(\ker A^{*})^{\perp}=(\ker A)^{\perp}=W^{\perp}\,.$$
\end{proof}

\begin{lemma}\label{lem:red}
	Assume \textup{(H1)} and \textup{(H2)}. Then:
	\begin{enumerate}
		\item $0$ is a semisimple eigenvalue of $A$, i.e.\ $\ker A^{2}=\ker A$.
		In fact, independently of \textup{(H2)}, there is a linear isomorphism
		$$\ker A^{2}/\ker A\cong W\cap W^{\perp}\,,$$ so semisimplicity of $0$ is
		\emph{equivalent} to \textup{(H1)}.
		\item $V=W\oplus W^{\perp}$ and this decomposition is $A$-invariant;
		$A':=A|_{W^{\perp}}$ is invertible with
		$$\Spec(A')=\Spec(A)\setminus\{0\}\subseteq\C\setminus\R\,;$$
		$g':=g|_{W^{\perp}}$ is nondegenerate and $A'$ is $g'$-antiadjoint.
		\item The projection $\pi\colon V\to V/W$ restricts to a linear isomorphism
		$$\pi|_{W^{\perp}}\colon W^{\perp}\to V/W$$ intertwining $A'$ with the
		operator $\bar A$ induced by $A$ on $V/W$. The reduced form
		$\bar\omega(\bar v,\bar w):=\omega(v,w)$ is a well-defined
		\emph{nondegenerate} symplectic form on $V/W$, and
		$\bar\omega(\pi u,\pi u')=g'(u,A'u')$ for $u,u'\in W^{\perp}$.
		\item \textup{(H1)}\,$\wedge$\,\textup{(H2)} $\iff$
		$\bigl[\,W\neq\{0\}$ and $\Spec(\bar A)\cap\R=\varnothing\,\bigr]$.
	\end{enumerate}
\end{lemma}

\begin{proof}
	\emph{(1).} Define $$\varphi\colon\ker A^{2}\to\ker A\cap\im A$$ by
	$\varphi(x):=Ax$; it is well defined because $A(Ax)=0$, its kernel is
	$\ker A$, and it is surjective: if $y\in\ker A\cap\im A$, say $y=Ax$, then
	$A^{2}x=Ay=0$, so $x\in\ker A^{2}$ and $\varphi(x)=y$. By
	Lemma~\ref{lem:kerim}, $$\ker A\cap\im A=W\cap W^{\perp},$$ whence
	$$\ker A^{2}/\ker A\cong W\cap W^{\perp}.$$ Thus \textup{(H1)} is equivalent to
	$\ker A^{2}=\ker A$. Finally, $\ker A^{2}=\ker A$ implies
	$\ker A^{m}=\ker A$ for all $m\geq1$ by induction. %if $x\in\ker A^{m+1}$ then	$$A^{m-1}x\in\ker A^{2}=\ker A.$$ so $x\in\ker A^{m}=\ker A$. Hence the	generalized $0$-eigenspace equals $\ker A$, i.e.\ $0$ is semisimple.
	
	\emph{(2).} Nondegeneracy of $g$ gives $\dim W^{\perp}=\dim V-\dim W$, so
	\textup{(H1)} yields $V=W\oplus W^{\perp}$. Then, we check the invariance: $A(W)=\{0\}\subseteq W$
	by Lemma~\ref{lem:kerim}; and for $u\in W^{\perp}$, $w\in W$ we get
	$$g(Au,w)=-g(u,Aw)=0,$$ so $A(W^{\perp})\subseteq W^{\perp}$. 
	
	The radical of
	$g'$ is $$W^{\perp}\cap(W^{\perp})^{\perp}=W^{\perp}\cap W=\{0\}$$ (using
	$(W^{\perp})^{\perp}=W$), so $g'$ is nondegenerate, and antiadjointness of
	$A'$ is inherited from $A^{*}=-A$. Next,
	$$\ker A'=\ker A\cap W^{\perp}=W\cap W^{\perp}=\{0\}\,,$$ so $A'$ is invertible.
	In a basis adapted to $V=W\oplus W^{\perp}$ the operator $A$ is block diagonal
	with blocks $0$ and $A'$, so the characteristic polynomials satisfy
	$$p_{A}(\lambda)=\lambda^{\dim W}p_{A'}(\lambda);$$ since $0\notin\Spec(A')$,
	this gives $\Spec(A')=\Spec(A)\setminus\{0\}$, which avoids $\R$ by
	\textup{(H2)}.
	
	\emph{(3).} $\pi|_{W^{\perp}}$ has kernel $W\cap W^{\perp}=\{0\}$ and the
	dimensions agree, so it is an isomorphism; the intertwining relation
	$\pi\circ A'=\bar A\circ\pi|_{W^{\perp}}$ is immediate. The form $\bar\omega$
	is well defined because $W=\rad\omega$, and it is nondegenerate by
	construction. For $u,u'\in W^{\perp}$,
	$$\bar\omega(\pi u,\pi u')=\omega(u,u')=g(u,Au')=g'(u,A'u').$$
	
	\emph{(4).} With respect to \emph{any} complement of $W$ the operator $A$ is
	block triangular with vanishing block on $W$, so
	$$p_{A}(\lambda)=\lambda^{\dim W}p_{\bar A}(\lambda).$$ Hence
	$0\in\Spec(\bar A)$ iff the algebraic multiplicity of $0$ in $A$ exceeds the dimension
	$\dim W=\dim\ker A$, i.e.\ iff $0$ is not semisimple, i.e.\ (by (1)) iff
	\textup{(H1)} fails; and the nonzero real eigenvalues of $\bar A$ are exactly
	those of $A$. The claimed equivalence follows. %(the condition $W\neq\{0\}$ accounting for the requirement $0\in\Spec(A)$ contained in \textup{(H2)}).
\end{proof}

Now, we would like to adapt the proof of \cite[Proposition 3.4, Lecture 1]{w}, see also \cite{ms}, to the Pseudo-Euclidean setting. Thus, we need to review the property of the holomorphic functional calculus.

\begin{lemma}\label{lem:calc}
	Let $T\in\End(V_{\C})$, let $U\subseteq\C$ be open with $\Spec(T)\subseteq U$,
	and let $\Gamma\subset U\setminus\Spec(T)$ be a finite union of positively
	oriented contours winding once around every point of $\Spec(T)$ and zero
	times around every point of $\C\setminus U$. For $\varphi$ holomorphic on $U$
	set
	\[
	\varphi(T):=\frac{1}{2\pi i}\oint_{\Gamma}\varphi(z)\,(z\Id-T)^{-1}\,dz .
	\]
	Then:
	\begin{enumerate}
		\item $\varphi\mapsto\varphi(T)$ is a unital algebra homomorphism extending
		the polynomial calculus; in particular every identity between
		holomorphic functions on $U$ passes to the corresponding operators,
		and $\varphi(T)$ commutes with every operator commuting with $T$;
		\item  $S\,\varphi(T)\,S^{-1}=\varphi(STS^{-1})$ for
		every invertible $S$;
		\item if $b$ is a nondegenerate \emph{bilinear} form on
		$V_{\C}$ and $T^{\star}$ denotes the $b$-adjoint, then
		$\varphi(T)^{\star}=\varphi(T^{\star})$;
		\item if $T$ is real, $U$ is stable under conjugation and
		$\varphi(\bar z)=\overline{\varphi(z)}$, then $\varphi(T)$ is real;
		\item if
		$\Spec(T)=S_{1}\sqcup\dots\sqcup S_{k}$ is a partition into relatively
		open closed subsets, with Riesz projections
		$P_{j}$, and if
		$\varphi\equiv c_{j}$ on a neighborhood of $S_{j}$, then
		$\varphi(T)=\sum_{j}c_{j}P_{j}$.
	\end{enumerate}
\end{lemma}

\begin{proof}
	(1), (2) and (5) are the standard theory; see e.g.\ \cite[Ch.~VII]{DS} or \cite[Ch.~1]{Hig}. For (3), note that
	$$b\bigl((z\Id-T)^{-1}v,w\bigr)=b\bigl(v,(z\Id-T^{\star})^{-1}w\bigr)\,,$$ since
	$$b\bigl((z\Id-T)x,y\bigr)=b\bigl(x,(z\Id-T^{\star})y\bigr)$$ %(no conjugationvoccurs: $b$ is bilinear); 
	integrating over $\Gamma$, which is admissible for
	$T^{\star}$ as well because $$\Spec(T^{\star})=\Spec(T),$$ gives the claim. For
	(4), $\sigma T\sigma=T$ and conjugate-linearity of $\sigma$ give
	$$\sigma(z\Id-T)^{-1}\sigma=(\bar z\Id-T)^{-1}\,;$$ conjugating the integral and
	substituting $z\mapsto\bar z$ (which reverses the orientation of the contour
	and conjugates the factor $(2\pi i)^{-1}$) yields
	$\sigma\varphi(T)\sigma=\varphi_{\dagger}(T)$ with
	$\varphi_{\dagger}(z):=\overline{\varphi(\bar z)}$; hence $\varphi(T)$ is real
	whenever $\varphi_{\dagger}=\varphi$.
\end{proof}

Eventually, we are ready to state the main theorem of this section. First, we state the following simple Lemma which collects some properties of the function $f(z):=z\,(-z^{2})^{-1/2}$. It will allow us to define $A'\,(-A'^{2})^{-1/2}$ as in the standard Euclidean case, see \cite{w} or \cite{ms}.

\begin{lemma}\label{lem:f}
	Let $f(z):=z\,(-z^{2})^{-1/2}$, where $(\,\cdot\,)^{1/2}$ denotes the
	principal branch of the square root, holomorphic on
	$\C\setminus(-\infty,0]$. Then:
	\begin{enumerate}
		\item $f$ is holomorphic on $\C\setminus\R$, and this domain is maximal:
		$f$ does not extend holomorphically (nor continuously) across any
		real point;
		\item $f(z)^{2}=-1$ and $f(-z)=-f(z)$ for all $z\in\C\setminus\R$;
		\item $f(\bar z)=\overline{f(z)}$;
		\item $f$ is locally constant: $f\equiv i$ on $\{\operatorname{Im}z>0\}$ and
		$f\equiv-i$ on $\{\operatorname{Im}z<0\}$.
	\end{enumerate}
\end{lemma}

%\begin{proof}	For $z\in\C\setminus\R$ we have $z^{2}\in\C\setminus[0,+\infty)$, hence	$-z^{2}\in\C\setminus(-\infty,0]$, and the principal square root of $-z^{2}$	is holomorphic and nonvanishing there; this proves the first part of (1).	Then $f(z)^{2}=z^{2}/(-z^{2})=-1$ and $f(-z)=(-z)(-z^{2})^{-1/2}=-f(z)$,	which is (2). Since $-z^{2}\in\C\setminus(-\infty,0]$, the principal branch	satisfies $\overline{(-z^{2})^{1/2}}=(-\bar z^{2})^{1/2}$, whence	$f(\bar z)=\overline{f(z)}$, i.e.\ (3). By (2) and continuity, $f$ takes	values in $\{\pm i\}$ and is therefore constant on each of the two connected	half-planes; evaluating $f(i)=i/\sqrt{1}=i$ and $f(-i)=-i$ gives (4).	Maximality in (1): by (3)--(4) the boundary values from the two sides of $\R$	are $i$ and $-i$, so no continuous extension across $\R$ exists.\end{proof}

\begin{theorem}[Canonical complex structure on $V/W$]\label{thm:main}
	Assume \textup{(H1)} and \textup{(H2)}, and let $A'$, $g'$, $\bar A$,
	$\bar\omega$ be as in Lemma~\ref{lem:red}. Set
	\[
	J:=f(A')=A'\,(-A'^{2})^{-1/2}\in\End(W^{\perp}).
	\]
	Then $J$ is well defined and real, and:
	\begin{enumerate}
		\item $J^{2}=-\Id$; thus $J$ is a complex structure on
		$W^{\perp}\cong V/W$;
		\item $J^{*}=-J$ with respect to $g'$; in particular $J\in\Ort(g')$;
		\item $J$ preserves the reduced symplectic form:
		$\bar\omega$ transported to $W^{\perp}$ satisfies
		$g'(Ju,A'Ju')=g'(u,A'u')$ for all $u,u'\in W^{\perp}$, i.e.\
		$J\in\Sp(\bar\omega)$ under the identification of
		Lemma~\ref{lem:red}(3);
		\item with $R:=(-A'^{2})^{1/2}$ \textup(principal branch\textup), the
		operator $R$ is real, $g'$-selfadjoint, invertible, commutes with
		$A'$ and $J$, and
		\[
		A'=JR=RJ
		\]
		\textup(a polar-type decomposition of $A'$\textup);
		\item the bilinear form
		\[
		G(u,u'):=\bar\omega(Ju,u')=g'(Ru,u'),\qquad u,u'\in W^{\perp},
		\]
		is symmetric, nondegenerate and $J$-invariant, and
		$\sgn G=\sgn g'$; in particular $G$ is positive definite if and only
		if $g|_{W^{\perp}}$ is;
		\item $J=i\,(P_{+}-P_{-})$, where $P_{\pm}$ are the Riesz projections of
		$A'$ associated with the spectral subsets
		$\{\lambda\in\Spec(A'):\pm\operatorname{Im}\lambda>0\}$.
		Consequently the $(+i)$-eigenspace of $J$ in $(W^{\perp})_{\C}$ is
		\[
		E_{+}:=\bigoplus_{\substack{\lambda\in\Spec(A')\\ \operatorname{Im}\lambda>0}}
		\ker\bigl(A'-\lambda\bigr)^{\dim V},
		\]
		and $J$ is the unique real complex structure on $W^{\perp}$ with this
		$(+i)$-eigenspace.
	\end{enumerate}
	All of the above transports to $V/W$ through $\pi|_{W^{\perp}}$; the
	resulting complex structure on $(V/W,\bar\omega)$ is $\bar J=f(\bar A)$, it
	satisfies $\bar\omega(\bar J\cdot,\bar J\cdot)=\bar\omega$, and
	$\bar G:=\bar\omega(\bar J\cdot,\cdot)$ is a nondegenerate symmetric
	$\bar J$-invariant form with $\sgn\bar G=\sgn\bigl(g|_{W^{\perp}}\bigr)$.
\end{theorem}

\begin{proof}
	By Lemma~\ref{lem:red}(2), $\Spec(A')\subseteq\C\setminus\R$, and $\C\setminus\R$ is the domain of $f$ (Lemma~\ref{lem:f}(1)). Hence $J=f(A')$ is well defined. Since $A'$ is real, the domain $\C\setminus\R$ is stable under conjugation and
	$f(\bar z)=\overline{f(z)}$ (Lemma~\ref{lem:f}(3)), then $J$ is real by	Lemma~\ref{lem:calc}(4).
	
	\emph{(1).} By Lemma~\ref{lem:calc}(1) and $f^{2}\equiv-1$
	(Lemma~\ref{lem:f}(2)), we have $$J^{2}=(f^{2})(A')=-\Id.$$
	
	\emph{(2).} By Lemma~\ref{lem:calc}(3) applied to $b=g'$ (extended
	bilinearly) and by $A'^{*}=-A'$,
	\[
	J^{*}=f(A'^{*})=f(-A')=-f(A')=-J,
	\]
	where the middle equality uses that $f$ is odd (Lemma~\ref{lem:f}(2)); note
	that $$\Spec(-A')=-\Spec(A')\subseteq\C\setminus\R\,,$$ so $f(-A')$ is defined.
	Then $J^{*}J=-J^{2}=\Id$, i.e.\ $J\in\Ort(g')$.
	
	\emph{(3).} $J$ commutes with $A'$ by Lemma~\ref{lem:calc}(1). Hence, for
	$u,u'\in W^{\perp}$,
	\[
	g'(Ju,A'Ju')=g'(Ju,JA'u')=g'(u,J^{*}JA'u')=g'(u,A'u'),
	\]
	which, by Lemma~\ref{lem:red}(3), is exactly the statement
	$\bar\omega(\bar J\bar u,\bar J\bar u')=\bar\omega(\bar u,\bar u')$.
	
	\emph{(4).} For $\lambda\in\Spec(A')\subseteq\C\setminus\R$ one has
	$$-\lambda^{2}\in(-\infty,0]\iff\lambda^{2}\in[0,+\infty)\iff\lambda\in\R\,,$$
	which is excluded; hence
	$$\Spec(-A'^{2})=\{-\lambda^{2}:\lambda\in\Spec(A')\}$$ avoids $(-\infty,0]$
	and the principal square root $\psi(z)=z^{1/2}$ is holomorphic on a
	conjugation-stable neighborhood of it. Thus $R=\psi(-A'^{2})$ is defined,
	invertible ($0\notin\Spec(-A'^{2})$, and $\psi$ never vanishes there), real
	(Lemma~\ref{lem:calc}(4), since $\psi(\bar z)=\overline{\psi(z)}$ on
	$\C\setminus(-\infty,0]$), and $g'$-selfadjoint:
	$$R^{*}=\psi\bigl((-A'^{2})^{*}\bigr)=\psi\bigl(-(A'^{*})^{2}\bigr)
	=\psi(-A'^{2})=R$$ by Lemma~\ref{lem:calc}(3). It commutes with $A'$ and $J$	by Lemma~\ref{lem:calc}(1). 
	
	Finally, the identity $f(z)\,(-z^{2})^{1/2}=z$, valid on $\C\setminus\R$, passes to operators:
	$JR=A'$, and $RJ=JR$ by commutation.
	
	\emph{(5).} Using $$g'(Ju,x)=g'(u,J^{*}x)=-g'(u,Jx),$$ the intertwining
	$JA'=A'J$ and $A'J=RJ^{2}=-R$ from (4):
	\begin{align*}
	G(u,u')&=\bar\omega(Ju,u')=g'(Ju,A'u')=-g'(u,JA'u') \\ &=-g'(u,A'Ju')
	=g'(u,Ru')=g'(Ru,u'),
	\end{align*}
	the last step by $R^{*}=R$. Symmetry follows from
	$$g'(Ru,u')=g'(u,Ru')=g'(Ru',u);$$ nondegeneracy from invertibility of $R$;
	$J$-invariance from
	$$G(Ju,Ju')=g'(RJu,Ju')=g'(JRu,Ju')=g'(Ru,J^{*}Ju')=G(u,u').$$
	{For the signature, let $S:=R^{1/2}$ be the square root of $R$ defined by the
		holomorphic functional calculus with the principal branch: this is legitimate
		since $\Spec(A')\cap\R=\emptyset$ gives
		$$\Spec(-A'^{2})\subseteq\C\setminus(-\infty,0],$$ whence
		$\Spec(R)\subseteq\{\operatorname{Re}z>0\}$. As for $R$, the operator $S$ is
		real, invertible and $g'$-selfadjoint by Lemma~\ref{lem:calc}(1),(3). Hence
		\[
		G(u,u')=g'(Ru,u')=g'(S^{2}u,u')=g'(Su,Su') ,
		\]
		so $G$ is congruent to $g'$ and $\sgn G=\sgn g'$ by Sylvester's law of inertia.}
	
	\emph{(6).} The partition of $\Spec(A')$ into its upper and lower parts is a
	partition into relatively open closed subsets, and $f$ is locally constant with values $i$, $-i$ on the
	corresponding half-planes (Lemma~\ref{lem:f}(4)). Lemma~\ref{lem:calc}(5)
	gives $J=iP_{+}-iP_{-}$ with $P_{+}+P_{-}=\Id$. The image of $P_{\pm}$ is the
	sum of the generalized eigenspaces with $\pm\operatorname{Im}\lambda>0$, so
	for $v=P_{+}v+P_{-}v$ one has $$Jv=iv\iff P_{-}v=0\iff v\in\im P_{+}=E_{+}\,.$$
	{Finally,} we note that a real complex structure $J_{0}$ on $W^{\perp}$ is determined by
	its $(+i)$-eigenspace $E\subseteq(W^{\perp})_{\C}$, because its
	$(-i)$-eigenspace is then $\sigma E$ and $(W^{\perp})_{\C}=E\oplus\sigma E$.
\end{proof}

\begin{remark}[The Euclidean case]\label{rem:euclid}
	If $g$ is positive definite, then $A$ is skew-symmetric, hence normal, so
	$\Spec(A)\subseteq i\R$ and $A$ is semisimple; thus \textup{(H1)} is
	automatic and \textup{(H2)} holds whenever $W\neq\{0\}$. Moreover $g'>0$ implies $G>0$ by Theorem~\ref{thm:main}(5): one recovers the classical statement that every
	symplectic vector space, together with an inner product, carries a canonical
	\emph{compatible} complex structure. In this case $A'^{*}=-A'$ gives
	$-A'^{2}=A'A'^{*}$, so $R=(A'A'^{*})^{1/2}$ and $A'=JR$ is precisely the
	polar decomposition of $A'$, matching the classical construction
	$J=(\sqrt{AA^{*}})^{-1}A$ (See \cite[Proposition 3.4, Lecture 1]{w}).
\end{remark}

\subsection{Pseudo-K\"ahler structures on co-adjoint orbits}\label{sec:orbits}

In this section, we apply the linear construction to co-adjoint orbits of Loeb-type Lie groups. Let $G$ be a Lie group with finite-dimensional Lie algebra $\g$, and let
$\kappa_{e}$ be a pseudo-Euclidean product on $\g$ which is $\ad(\g)$-invariant:
\begin{equation}\label{eq:inv}
	\kappa_{e}([\xi,\zeta],\eta)+\kappa_{e}(\zeta,[\xi,\eta])=0
	\qquad\forall\,\xi,\zeta,\eta\in\g .
\end{equation}
We assume in addition that $\kappa_{e}$ is $\Ad(G)$-invariant; if $G$ is
connected this follows from \eqref{eq:inv} by differentiation. We use
$\kappa_{e}$ to identify $\g^{*}\cong\g$; under this identification the
coadjoint action corresponds to the adjoint action and coadjoint orbits to
adjoint orbits. Orthogonals $(\cdot)^{\perp}$ and adjoints are taken with
respect to $\kappa_{e}$.

Fix $\mu\in\g$, $\mu\neq0$, and let
$\mathcal O:=\mathcal O_{\mu}=\Ad(G)\mu\subseteq\g$ be its orbit, endowed
with the smooth structure of $G/G_{\mu}$, where
$G_{\mu}=\{a\in G:\Ad_{a}\mu=\mu\}$ is the stabilizer, a closed subgroup with
Lie algebra $\g_{\mu}=\ker\ad_{\mu}$. For $\nu\in\mathcal O$ the tangent
space is
\[
T_{\nu}\mathcal O=\{[\xi,\nu]:\xi\in\g\}=\im\ad_{\nu}\subseteq\g .
\]
The Kirillov--Kostant--Souriau (KKS) form~\cite{Kirillov:1962,Kostant:1970,Souriau:1970} on $\mathcal O$ is, in the
convention fixed above and at the point $\nu$,
\begin{equation}\label{eq:kks}
	\omega_{\nu}\bigl([\xi,\nu],[\eta,\nu]\bigr)=-\kappa_{e}(\nu,[\xi,\eta]),
	\qquad\xi,\eta\in\g ,
\end{equation}
where $[\xi,\nu]$ is the value at $\nu$ of the fundamental vector field of
$\xi$; the form \eqref{eq:kks} is well defined, nondegenerate and closed
\cite{Kir}. It will be convenient to introduce the auxiliary bilinear form on
all of $\g$
\[
\Omega_{\nu}(\xi,\eta):=-\kappa_{e}(\nu,[\xi,\eta]),\qquad\xi,\eta\in\g .
\]

\begin{lemma}\label{lem:orbitalg}
	For every $\nu\in\g$:
	\begin{enumerate}
		\item $\Omega_{\nu}(\xi,\eta)=\kappa_{e}(\xi,[\nu,\eta])
		=\kappa_{e}(\xi,\ad_{\nu}\eta)$. Hence, taking
		$$(V,\omega,g)=(\g,\Omega_{\nu},\kappa_{e})$$ in
		Section~\ref{sec:linalg}, the operator defined by \eqref{eq:defA} is
		\[
		A=\ad_{\nu} .
		\]
		\item $\ad_{\nu}$ is $\kappa_{e}$-antiadjoint, and
		\[
		W=\rad\Omega_{\nu}=\ker\ad_{\nu}=\g_{\nu},\qquad
		W^{\perp}=\im\ad_{\nu}=\g_{\nu}^{\perp}=T_{\nu}\mathcal O .
		\]
		\item if $[\xi,\nu]=0$ then
		$\kappa_{e}(\nu,[\xi,\eta])=0$ for all $\eta$; so \eqref{eq:kks}
		depends only on the tangent vectors $[\xi,\nu]$, $[\eta,\nu]$.
		\item the map
		\[
		\beta_{\nu}\colon\g/\g_{\nu}\longrightarrow T_{\nu}\mathcal O,
		\qquad
		\beta_{\nu}(\bar\xi):=[\xi,\nu]=-\ad_{\nu}\xi ,
		\]
		is a linear isomorphism which intertwines the induced operator
		$\bar A_{\nu}$ on $\g/\g_{\nu}$ with
		$A'_{\nu}:=\ad_{\nu}|_{T_{\nu}\mathcal O}$, and which pulls the KKS
		form back to the reduced form:
		$\beta_{\nu}^{*}\omega_{\nu}=\bar\Omega_{\nu}$.
	\end{enumerate}
\end{lemma}

\begin{proof}
	\emph{(1).} It follows by the application of \eqref{eq:inv}. %with	$(\xi,\zeta,\eta)\mapsto(\xi,\nu,\eta)$,	$$\kappa_{e}([\xi,\nu],\eta)=-\kappa_{e}(\nu,[\xi,\eta]);$$ with	$(\xi,\zeta,\eta)\mapsto(\nu,\xi,\eta)$,	$\kappa_{e}([\nu,\xi],\eta)=-\kappa_{e}(\xi,[\nu,\eta])$. Hence	\[	\Omega_{\nu}(\xi,\eta)=-\kappa_{e}(\nu,[\xi,\eta])=\kappa_{e}([\xi,\nu],\eta)=-\kappa_{e}([\nu,\xi],\eta)=\kappa_{e}(\xi,[\nu,\eta]) .\]
	
	\emph{(2).} Antiadjointness is Lemma~\ref{lem:kerim} applied to
	$\Omega_{\nu}$; the identifications of $W$ and
	$W^{\perp}$ are Lemma~\ref{lem:kerim}, together with
	$T_{\nu}\mathcal O=\im\ad_{\nu}$.
	
	\emph{(3).} If $[\xi,\nu]=0$, then
	$$-\kappa_{e}(\nu,[\xi,\eta])=\kappa_{e}([\xi,\nu],\eta)=0$$ for every $\eta$.
	
	\emph{(4).} $\beta_{\nu}$ is well defined and injective because
	$\ker\ad_{\nu}=\g_{\nu}$, and surjective onto
	$T_{\nu}\mathcal O=\im\ad_{\nu}$. Furthermore,
	\[
	\beta_{\nu}(\bar A_{\nu}\bar\xi)
	=\beta_{\nu}\bigl(\overline{\ad_{\nu}\xi}\bigr)
	=-\ad_{\nu}(\ad_{\nu}\xi)
	=A'_{\nu}\bigl(-\ad_{\nu}\xi\bigr)
	=A'_{\nu}\beta_{\nu}(\bar\xi).
	\]
	Finally,
	$$\beta_{\nu}^{*}\omega_{\nu}(\bar\xi,\bar\eta)
	=\omega_{\nu}([\xi,\nu],[\eta,\nu])=-\kappa_{e}(\nu,[\xi,\eta])
	=\Omega_{\nu}(\xi,\eta)=\bar\Omega_{\nu}(\bar\xi,\bar\eta)\,.$$
\end{proof}

\medskip
\noindent\textbf{Hypotheses at $\mu$.} We assume the hypotheses of
Section~\ref{sec:linalg} for $A=\ad_{\mu}$:
\begin{itemize}
	\item[\textup{(K1)}] $\g_{\mu}\cap\g_{\mu}^{\perp}=\{0\}$;
	\item[\textup{(K2)}] $\Spec(\ad_{\mu})\cap\R=\{0\}$.
\end{itemize}

\begin{remark}\label{rem:zeroalways}
	Since $[\mu,\mu]=0$, one always has $\mu\in\g_{\mu}$, so
	$\g_{\mu}\neq\{0\}$ and $0\in\Spec(\ad_{\mu})$ for every $\mu\neq0$
	(similarly, the center of $\g$ is contained in every $\g_{\mu}$). Hence in
	the orbit setting the spectral hypothesis necessarily takes the form
	``$\Spec(\ad_{\mu})\cap\R=\{0\}$'', as in \textup{(K2)}. By
	Lemma~\ref{lem:red}(1), hypothesis \textup{(K1)} is equivalent to the
	$\kappa_{e}$-independent condition $\ker\ad_{\mu}^{2}=\ker\ad_{\mu}$.
\end{remark}

\begin{lemma}\label{lem:equiv}
	Let $a\in G$ and $\nu=\Ad_{a}\mu$. Then:
	\begin{enumerate}
		\item $\ad_{\nu}=\Ad_{a}\circ\ad_{\mu}\circ\Ad_{a}^{-1}$;
		\item $\Ad_{a}\in\Ort(\kappa_{e})$, and consequently
		$\g_{\nu}=\Ad_{a}\g_{\mu}$,
		$T_{\nu}\mathcal O=\g_{\nu}^{\perp}=\Ad_{a}\bigl(\g_{\mu}^{\perp}\bigr)$,
		$\Spec(\ad_{\nu})=\Spec(\ad_{\mu})$, and $0$ is semisimple for
		$\ad_{\nu}$ iff it is for $\ad_{\mu}$. In particular \textup{(K1)}
		and \textup{(K2)} hold at every point of $\mathcal O$.
	\end{enumerate}
\end{lemma}

\begin{proof}
	(1) is the identity $[\Ad_{a}\mu,x]=\Ad_{a}[\mu,\Ad_{a}^{-1}x]$, valid
	because $\Ad_{a}$ is a Lie algebra automorphism. (2): orthogonality of
	$\Ad_{a}$ is the assumed $\Ad(G)$-invariance of $\kappa_{e}$; the remaining
	statements follow from (1) by conjugation-invariance of kernels, images,
	spectra and Jordan structure, and from orthogonality for the statement about
	$\perp$.
\end{proof}

{The following construction is classical for compact semisimple Lie groups, where co-adjoint orbits carry a natural K\"ahler structure: see \cite[Chapter 8]{Besse}. %and, for an operator-theoretic treatment close in spirit to ours, \cite[\S1]{Rieffel:2009}. Both ingredients in the proof of integrability are classical when taken separately: Step 1 is the primary decomposition with respect to a derivation \cite{Humphreys}, and Step 2 is the classical criterion for the integrability of invariant almost	complex structures on homogeneous spaces, which goes back to \cite{Koszul:1955,Wang:1954}. 
}

\begin{theorem}[Invariant pseudo-K\"ahler structure on the orbit]\label{thm:orbit}
	Assume \textup{(K1)} and \textup{(K2)}. For $\nu\in\mathcal O$ set
	$A'_{\nu}:=\ad_{\nu}|_{T_{\nu}\mathcal O}\in\End(T_{\nu}\mathcal O)$ and
	\[
	J_{\nu}:=f(A'_{\nu})=A'_{\nu}\,\bigl(-A'^{2}_{\nu}\bigr)^{-1/2}.
	\]
	Then:
	\begin{enumerate}
		\item $J=(J_{\nu})_{\nu\in\mathcal O}$ is a well-defined, smooth,
		$G$-invariant almost complex structure on $\mathcal O$:
		$J^{2}=-\Id$ and $J_{\Ad_{a}\nu}=\Ad_{a}\circ J_{\nu}\circ\Ad_{a}^{-1}$
		for all $a\in G$;
		\item $J$ is compatible with the KKS symplectic form in the sense that
		$$\omega(JX,JY)=\omega(X,Y),$$ and
		\[
		\gamma(X,Y):=\omega(JX,Y)
		\]
		defines a $G$-invariant, $J$-invariant pseudo-Riemannian metric on
		$\mathcal O$ of constant signature
		$\sgn\bigl(\kappa_{e}|_{\g_{\mu}^{\perp}}\bigr)$; it is positive
		definite iff $\kappa_{e}|_{\g_{\mu}^{\perp}}$ is positive definite;
		\item $J$ is integrable. Hence $(\mathcal O,J)$ is a complex manifold and
		$(\mathcal O,\omega,J,\gamma)$ is a pseudo-K\"ahler manifold; it is
		K\"ahler iff $\kappa_{e}|_{\g_{\mu}^{\perp}}$ is positive definite;
		\item for each $\nu$, the $(+i)$-eigenspace of
		$(J_{\nu})_{\C}$ is the sum of the generalized eigenspaces of
		$\ad_{\nu}$ in $(T_{\nu}\mathcal O)_{\C}$ with positive imaginary
		part, and $J$ is the unique $G$-invariant almost complex structure
		with this property at one (equivalently, every) point. Moreover
		$J_{\nu}$ depends only on $\ad_{\nu}$: the product $\kappa_{e}$
		enters only through the identification $\g^{*}\cong\g$, the
		hypothesis \textup{(K1)} \textup(itself equivalent to
		$\ker\ad_{\mu}^{2}=\ker\ad_{\mu}$, cf.\
		Remark~\ref{rem:zeroalways}\textup) and the metric $\gamma$.
	\end{enumerate}
\end{theorem}

\begin{proof}
	\emph{(1) and (2)} First, we recall the pointwise construction. Fix $\nu\in\mathcal O$. By
	Lemma~\ref{lem:orbitalg}, the data $(V,\omega,g,A)=(\g,\Omega_{\nu},
	\kappa_{e},\ad_{\nu})$ fit the setup of Section~\ref{sec:linalg} with
	$W=\g_{\nu}$ and $W^{\perp}=T_{\nu}\mathcal O$; by Lemma~\ref{lem:equiv},
	hypotheses \textup{(H1)}--\textup{(H2)} hold at $\nu$. Theorem~\ref{thm:main}
	applied on $V/W=\g/\g_{\nu}$ yields the real complex structure
	$\bar J_{\nu}=f(\bar A_{\nu})$, which preserves the reduced form
	$\bar\Omega_{\nu}$ and for which
	$\bar G_{\nu}:=\bar\Omega_{\nu}(\bar J_{\nu}\cdot,\cdot)$ is symmetric,
	nondegenerate, $\bar J_{\nu}$-invariant, of signature
	$\sgn(\kappa_{e}|_{\g_{\nu}^{\perp}})$.
	
	Now transport through the isomorphism $\beta_{\nu}$ of
	Lemma~\ref{lem:orbitalg}(4). Since $\beta_{\nu}$ intertwines
	$\bar A_{\nu}$ with $A'_{\nu}$, Lemma~\ref{lem:calc}(2) gives
	$$\beta_{\nu}\circ f(\bar A_{\nu})\circ\beta_{\nu}^{-1}=f(A'_{\nu})=J_{\nu}\,;$$
	and since $\beta_{\nu}^{*}\omega_{\nu}=\bar\Omega_{\nu}\,,$ the operator
	$J_{\nu}$ preserves $\omega_{\nu}$, satisfies $J_{\nu}^{2}=-\Id$, and
	$\gamma_{\nu}:=\omega_{\nu}(J_{\nu}\cdot,\cdot)$ corresponds to
	$\bar G_{\nu}$, hence is symmetric, nondegenerate, $J_{\nu}$-invariant, of
	signature $\sgn(\kappa_{e}|_{\g_{\nu}^{\perp}})
	=\sgn(\kappa_{e}|_{\g_{\mu}^{\perp}})$, the last equality because
	$\Ad_{a}$ is a $\kappa_{e}$-isometry carrying $\g_{\mu}^{\perp}$ onto
	$\g_{\nu}^{\perp}$ (Lemma~\ref{lem:equiv}).
	
	Let $\nu=\Ad_{a}\mu$. By Lemma~\ref{lem:equiv}(1) and Lemma~\ref{lem:calc}(2),
	\[
	J_{\Ad_{a}\mu}=f\bigl(\Ad_{a}\,A'_{\mu}\,\Ad_{a}^{-1}\bigr)
	=\Ad_{a}\,f(A'_{\mu})\,\Ad_{a}^{-1}
	=\Ad_{a}\,J_{\mu}\,\Ad_{a}^{-1},
	\]
	where we used that $\Ad_{a}$ maps $T_{\mu}\mathcal O$ onto
	$T_{\nu}\mathcal O$. In particular, for $h\in G_{\mu}$ we get
	$\Ad_{h}J_{\mu}\Ad_{h}^{-1}=J_{\mu}$: the tensor $J_{\mu}$ is
	$G_{\mu}$-invariant, so the family $(J_{\nu})_{\nu}$ is a well-defined
	$G$-invariant tensor field on $\mathcal O\cong G/G_{\mu}$. The same argument
	applies to $\gamma$, whose $G$-invariance also follows from that of $\omega$
	(KKS) and of $J$.
	
	Eventually, we prove the smoothness of $J$. Choose pairwise disjoint open sets
	$U_{0}\ni0$, $$U_{+}\supseteq\Spec(\ad_{\mu})\cap\{\operatorname{Im}z>0\}\,,\quad U_{-}\supseteq\Spec(\ad_{\mu})\cap\{\operatorname{Im}z<0\}\,,$$ and define
	$h\in\mathcal O(U_{0}\cup U_{+}\cup U_{-})$ by $h\equiv0$ on $U_{0}$,
	$h\equiv i$ on $U_{+}$, $h\equiv-i$ on $U_{-}$. By Lemma~\ref{lem:equiv},
	$\Spec(\ad_{\nu})=\Spec(\ad_{\mu})$ for every $\nu\in\mathcal O$, so a single
	finite contour system $$\Gamma\subset(U_{0}\cup U_{+}\cup U_{-})	\setminus\Spec(\ad_{\mu})$$ computes
	\[
	H_{\nu}:=h(\ad_{\nu})
	=\frac{1}{2\pi i}\oint_{\Gamma}h(z)\,(z\Id-\ad_{\nu})^{-1}\,dz
	\in\End(\g)
	\]
	simultaneously for all $\nu\in\mathcal O$. The map $\nu\mapsto\ad_{\nu}$ is
	linear, hence smooth; $\Gamma$ is compact and contained in the resolvent set
	of every $\ad_{\nu}$, and matrix inversion is smooth, so
	$(z,\nu)\mapsto(z\Id-\ad_{\nu})^{-1}$ is smooth on
	$\Gamma\times\mathcal O$ and $\nu\mapsto H_{\nu}$ is a smooth
	$\End(\g)$-valued map. By Lemma~\ref{lem:calc}(5),
	$H_{\nu}=i\bigl(P_{+}(\nu)-P_{-}(\nu)\bigr)$, where $P_{0},P_{\pm}$ are the
	Riesz projections of $\ad_{\nu}$ for the spectral subsets $\{0\}$ and
	$\{\pm\operatorname{Im}>0\}$; hence $H_{\nu}$ vanishes on
	$(\g_{\nu})_{\C}=\im P_{0}(\nu)$ (semisimplicity of $0$,
	Lemma~\ref{lem:red}(1)) and restricts on
	$(T_{\nu}\mathcal O)_{\C}=\im P_{+}(\nu)\oplus\im P_{-}(\nu)$ (the
	complexification of Lemma~\ref{lem:red}(2)) to $f(A'_{\nu})=J_{\nu}$,
	because $h\equiv f$ on $U_{+}\cup U_{-}$. Therefore
	$J=H|_{T\mathcal O}$ is a smooth bundle endomorphism. This proves (1) and
	(2).
	
	\emph{(3) Integrability.} We use the following two ingredients.
	
	\medskip
	\noindent\textbf{Step 1 (Leibniz rule for generalized eigenspaces).}
	\emph{Let $D$ be a derivation of a complex Lie algebra $\mathfrak l$ and let
		$$E_{\lambda}:=\bigcup_{m}\ker(D-\lambda)^{m}$$ denote its generalized
		eigenspaces. Then $[E_{\lambda},E_{\rho}]\subseteq E_{\lambda+\rho}$.}
	
	Indeed, by induction on $m$ one proves
	\[
	(D-\lambda-\rho)^{m}[x,y]
	=\sum_{j=0}^{m}\binom{m}{j}
	\bigl[(D-\lambda)^{j}x,\;(D-\rho)^{m-j}y\bigr]\,.
	\]
	The case $m=0$ is trivial, and applying $D-\lambda-\rho$ to a summand and
	using $D[a,b]=[Da,b]+[a,Db]$ gives
	\begin{multline*}
	(D-\lambda-\rho)\bigl[(D-\lambda)^{j}x,(D-\rho)^{m-j}y\bigr]
	\\=\bigl[(D-\lambda)^{j+1}x,(D-\rho)^{m-j}y\bigr]
	+\bigl[(D-\lambda)^{j}x,(D-\rho)^{m-j+1}y\bigr],
	\end{multline*}
	so the inductive step follows from Pascal's rule. If now
	$(D-\lambda)^{p}x=0$ and $(D-\rho)^{q}y=0$, take $m=p+q-1$: in every summand
	either $j\geq p$ or $m-j\geq q$, so $(D-\lambda-\rho)^{p+q-1}[x,y]=0$.
	
	\medskip
	\noindent\textbf{Step 2 (integrability criterion).}
	\emph{Let $M=G/H$ with base point $o$, let $J$ be a $G$-invariant almost
		complex structure, let $q\colon\g_{\C}\to(\g/\mathfrak h)_{\C}$ be the
		complexified quotient map, let
		$\mathfrak m^{+}\subseteq(\g/\mathfrak h)_{\C}$ be the $(+i)$-eigenspace of
		$J_{o}$, and set $\mathfrak q:=q^{-1}(\mathfrak m^{+})$. If $\mathfrak q$ is
		a Lie subalgebra of $\g_{\C}$, then $J$ is integrable, i.e.\ its Nijenhuis
		tensor vanishes and $J$ underlies a complex manifold structure on $M$.}
	
	%This criterion is classical: by $G$-invariance, the vanishing of the Nijenhuis tensor reduces to the algebraic condition	$[\mathfrak q,\mathfrak q]\subseteq\mathfrak q$ at the base point	(Fr\"olicher \cite{Fro}); the passage from $N_{J}=0$ to an underlying complex	structure is the Newlander--Nirenberg theorem \cite{NN} (in the present	real-analytic homogeneous situation Fr\"olicher's original argument	suffices).

	This criterion is classical (Fr\"olicher \cite{Fro}; via the Newlander--Nirenberg theorem \cite{NN}); we apply Step 2 to $M=\mathcal O\cong G/G_{\mu}$, $o=\mu$,
	$\mathfrak h=\g_{\mu}$. Write
	$\g_{\C}=E_{0}\oplus E_{+}\oplus E_{-}$ for the decomposition of $\g_{\C}$
	into the generalized eigenspaces of $\ad_{\mu}$ grouped by
	$\lambda=0$, $\operatorname{Im}\lambda>0$, $\operatorname{Im}\lambda<0$
	(exhaustive by \textup{(K2)}). By semisimplicity of $0$
	(Lemma~\ref{lem:red}(1)), $E_{0}=(\g_{\mu})_{\C}=\ker q$. The differential at
	$o$ of the identification $G/G_{\mu}\cong\mathcal O$ is exactly
	$\beta_{\mu}$; since $\beta_{\mu}$ intertwines $\bar A_{\mu}$ with
	$A'_{\mu}$, and since by Theorem~\ref{thm:main}(6) the $(+i)$-eigenspace of
	$J_{\mu}$ in $(T_{\mu}\mathcal O)_{\C}$ is the sum of the generalized
	eigenspaces of $A'_{\mu}$ with positive imaginary part, i.e.\
	$E_{+}\subseteq(T_{\mu}\mathcal O)_{\C}$, the corresponding
	$(+i)$-eigenspace of $J_{o}$ on $(\g/\g_{\mu})_{\C}$ is
	$\mathfrak m^{+}=q(E_{+})$ (note $\beta_{\mu}(q(x))=-\ad_{\mu}x\in
	E_{\lambda}$ for $x\in E_{\lambda}$, so $\beta_{\mu}$ matches the two
	gradings). Hence
	\[
	\mathfrak q=q^{-1}\bigl(q(E_{+})\bigr)=E_{0}\oplus E_{+}.
	\]
	Now $\ad_{\mu}$ is a derivation of $\g_{\C}$ (Jacobi identity), so Step 1
	gives $[E_{\lambda},E_{\rho}]\subseteq E_{\lambda+\rho}$. Let
	$\lambda,\rho\in\{0\}\cup\{\operatorname{Im}>0\}$-part of the spectrum. Then
	$\operatorname{Im}(\lambda+\rho)=\operatorname{Im}\lambda
	+\operatorname{Im}\rho\geq0$, with equality iff $\lambda=\rho=0$ (if
	$\lambda+\rho$ were a nonzero real number it would be a nonzero real point
	of $\Spec(\ad_{\mu})$, excluded by \textup{(K2)}; and $\lambda+\rho=0$ with
	both imaginary parts $\geq0$ forces $\lambda,\rho\in\R$, hence
	$\lambda=\rho=0$). Therefore $E_{\lambda+\rho}\subseteq E_{0}\oplus E_{+}$
	in all cases (it is $\{0\}$ if $\lambda+\rho\notin\Spec(\ad_{\mu})$), and
	$[\mathfrak q,\mathfrak q]\subseteq\mathfrak q$. By Step 2, $J$ is
	integrable.
	
	Finally, $d\omega=0$ (see \cite{Kir} for example), $\omega(J\cdot,J\cdot)=\omega$,
	$N_{J}=0$ and $\gamma=\omega(J\cdot,\cdot)$ nondegenerate symmetric imply
	that $(\mathcal O,\omega,J,\gamma)$ is pseudo-K\"ahler, i.e.\
	$\nabla J=0$ for the Levi-Civita connection of $\gamma$: the classical proof
	of the equivalence ``$d\omega=0$ and $N_{J}=0$ $\iff$ $\nabla J=0$'' does not
	use definiteness of the metric. It is K\"ahler precisely when $\gamma>0$,
	i.e.\ when $\kappa_{e}|_{\g_{\mu}^{\perp}}>0$ by (2).
	
	\emph{(4).} The description of the $(+i)$-eigenspace is
	Theorem~\ref{thm:main}(6) transported by $\beta_{\nu}$ as above; a
	$G$-invariant tensor is determined by its value at one point, which gives
	uniqueness. For the last statement: $T_{\nu}\mathcal O=\im\ad_{\nu}$ and
	$A'_{\nu}=\ad_{\nu}|_{\im\ad_{\nu}}$ are defined by $\ad_{\nu}$ alone, and
	$J_{\nu}=f(A'_{\nu})$; $\kappa_{e}$ is used to write $\omega$ via the
	identification $\g^{*}\cong\g$, to formulate \textup{(K1)} (equivalent to
	$\ker\ad_{\mu}^{2}=\ker\ad_{\mu}$ by Lemma~\ref{lem:red}(1) and
	Lemma~\ref{lem:orbitalg}(2)) and to produce $\gamma$, whose signature is
	$\sgn(\kappa_{e}|_{\g_{\mu}^{\perp}})$.
\end{proof}

\begin{example}[Compact case]\label{ex:compact}
	Let $G$ be compact and connected and let $\kappa_{e}$ be any
	$\Ad(G)$-invariant inner product (e.g.\ $-$Killing for $G$ semisimple). Then
	every $\ad_{\xi}$ is $\kappa_{e}$-skew, hence normal, so
	$\Spec(\ad_{\mu})\subseteq i\R$ and $\ad_{\mu}$ is semisimple: \textup{(K1)}
	and \textup{(K2)} hold for every $\mu\neq0$. Since
	$\kappa_{e}|_{\g_{\mu}^{\perp}}>0$, Theorem~\ref{thm:orbit} produces the
	classical $G$-invariant K\"ahler structure on (co)adjoint orbits; for $\mu$
	in the interior of a Weyl chamber of a maximal torus, $E_{+}$ is the sum of
	the root spaces on which the corresponding roots are positive. For explicit invariant \emph{indefinite}
	pseudo-K\"ahler metrics on flag manifolds see also \cite{MasonZiegler}.
\end{example}

\section{The pseudo-K\"ahler reduction}
\label{sec:psN}

\subsection{Preliminaries} \label{sec:polar}

In this section we make the pseudo-K\"ahler structure of Theorem \ref{thm:Szoke} completely explicit. For the general theory of adapted complex structures and of the polar map we refer to \cite{Szoke}. Throughout, left trivialization and the isomorphism $\lieg^*\cong\lieg$ induced by $\ka_e$ identify
\begin{equation}\label{eq:ident}
	T^*G\;\cong\;TG\;\cong\;G\times\lieg,\qquad p\longleftrightarrow(q,x),\quad
	x=q^{-1}p,
\end{equation}
tangent vectors at $(q,x)$ being pairs $(\xi,\eta)\in\lieg\oplus\lieg$ with $\xi=q^{-1}\dot q$ and $\eta=\dot x$; under \eqref{eq:ident} the canonical symplectic structure $\Omega_{\mathrm{can}}$ of $TG$ corresponds to $\omega=d\varpi$, and the map $E$ of the previous section becomes the polar map $P$ below, so that the complex structure computed here is precisely the adapted structure $\Jad$ of Theorem \ref{thm:Szoke}. Every linear operator occurring below is an analytic function $F(A)$ of the single $\ka_e$-skew operator $A=\ad_x$, interpreted through the holomorphic calculus of Lemma \ref{lem:calc}. Recall that
\begin{equation}\label{eq:symbolrules}
	\ka\bigl(F(A)\xi,\eta\bigr)=\ka\bigl(\xi,F(-A)\eta\bigr),
	\qquad F(A)G(A)=(FG)(A)\,,
\end{equation}
and a $\ka$-symmetric $F(A)$ is invertible if and only if $F$ has no zero on $\Spec(A)$.

Introduce the four functions
\begin{equation}\label{eq:symbols}
	\esym(z)= \frac{\sin z}{z},\qquad
	\fsym(z)=z\cot z,\qquad
	\tsym(z)=\tan\frac z2,\qquad
	\vsym(z)=\frac2z\tan\frac z2,
\end{equation}
so that $	z\,\vsym(z)=2\,\tsym(z)$,
with $$\lim_{z\rightarrow 0}\esym(z)=\lim_{z\rightarrow 0}\fsym(z)=\lim_{z\rightarrow 0}\vsym(z)=1\,,$$ $\tsym(0)=0\,.$

\begin{lemma}\label{lem:loeb}
	The zeros of $\esym$ and the poles of
	$\fsym,\tsym,\vsym$ lie in $\pi(\Z\setminus\{0\})\subset\R^*$, and the
	zeros of $\fsym$ lie in $\frac\pi2+\pi\Z\subset\R^*$. Hence if $G$ is of
	Loeb type, then for \emph{every} $x\in\lieg$ the operators
	$\esym(\ad_x)$, $\fsym(\ad_x)$, $\tsym(\ad_x)$, $\vsym(\ad_x)$ are defined, and
	$\esym({\ad_x})$ and $\fsym(\ad_x)$ are invertible.
\end{lemma}

\begin{proof}
	$\sin z=0$ iff $z\in\pi\Z$; $\cos z=0$ iff $z\in\frac\pi2+\pi\Z$; both
	sets are real, and a Loeb-type spectrum meets $\R$ only in $0$.
\end{proof}

With $\tilde G$ the complexification of $G$ as above, write again
\[
P:\ G\times\lieg\longrightarrow \tilde G,\qquad
P(q,x)=q\,\exp_{\tilde G}(ix),
\]
for the polar map. Since $G$ is of Loeb type, Theorem \ref{thm:Szoke} shows that $P$ is a global diffeomorphism onto its image and that
\[
\Jad:=P^*(J_{\tilde G})
\]
is defined on all of $G\times\lieg\cong T^*G$, since the Grauert tube is entire.

\begin{lemma}\label{lem:dP}
	In the complex left trivialization of $TG_\C$,
	\begin{equation}\label{eq:dP}
		P(q,x)^{-1}\,dP_{(q,x)}(\xi,\eta)
		=e^{-i\ad_x}\,\xi+\frac{1-e^{-i\ad_x}}{\ad_x}\,\eta ,
	\end{equation}
	and $P$ is a local diffeomorphism at $(q,x)$ if and only if
	$\esym( \ad_x)$ is invertible. For Loeb-type $G$ this holds at every point, by Lemma \ref{lem:loeb}.
\end{lemma}

\begin{proof}
	Equation \eqref{eq:dP} follows from $\Ad_{\exp(-ix)}=e^{-i\ad_x}$ and the
	differential-of-exponential formula (see e.g.\ \cite{DuistermaatKolk})
	$$\exp(X)^{-1}d\exp_X(\dot X)=\frac{1-e^{-\ad_X}}{\ad_X}\dot X$$ with
	$X=ix$, $\dot X=i\eta$. Writing $e^{-iz}=\mathsf C-i \mathsf S$ and
	$$\frac{1-e^{-iz}}{z}=\mathsf P+i\mathsf Q$$ with
	\[
	\mathsf C=\cos z,\quad \mathsf S=\sin z,\quad \mathsf P=\frac{1-\cos z}{z},
	\quad \mathsf Q=\frac{\sin z}{z},
	\]
	the real form of \eqref{eq:dP} has block matrix
	$\bigl(\begin{smallmatrix} \mathsf C&\mathsf P\\ - \mathsf S&\mathsf Q\end{smallmatrix}\bigr)${, the symbols being evaluated at $z=\ad_x$ through the holomorphic calculus of Lemma \ref{lem:calc},}
	whose determinant is $\mathsf C\mathsf Q+ \mathsf S\mathsf P=\esym(z)$.
\end{proof}

\begin{proposition}\label{prop:J}
	At $(q,x)$, in the trivialization \eqref{eq:ident},
	\begin{equation}\label{eq:J}
		\Jad(\xi,\eta)=\Bigl(\ \tsym(\ad_x)\,\xi-\vsym(\ad_x)\,\eta,\ \
		\frac{1}{\esym(\ad_x)}\,\xi-\tsym(\ad_x)\,\eta\ \Bigr).
	\end{equation}
\end{proposition}

\begin{proof}
	Write $\Jad(\xi,\eta)=(\tilde\xi,\tilde\eta)$; by definition
	$dP(\tilde\xi,\tilde\eta)=i\,dP(\xi,\eta)$, i.e.
	\[
	\begin{pmatrix}\mathsf C&\mathsf P\\ -\mathsf S&\mathsf Q\end{pmatrix}
	\begin{pmatrix}\tilde\xi\\ \tilde\eta\end{pmatrix}
	=\begin{pmatrix}\mathsf S&-\mathsf Q\\ \mathsf C&\mathsf P\end{pmatrix}
	\begin{pmatrix}\xi\\ \eta\end{pmatrix},
	\]
	which can be solved by Lemma \ref{lem:dP}. Using
	$$\mathsf Q \mathsf S-\mathsf P \mathsf C=\mathsf P\,,\quad \mathsf Q^2+\mathsf P^2=\frac{2(1-\mathsf C)}{z^2}\,,\quad
	\mathsf C\mathsf P-\mathsf S\mathsf Q=-\mathsf P$$ and
	$$\frac{1-\cos z}{\sin z}=\tan\frac z2\,,$$ the solution simplifies to
	\eqref{eq:J}. Finally, $\Jad^2=-\id$ follows from
	$$\tsym^2(z)-\vsym(z)\cdot\frac{z}{\sin z}
	=\tan^2\frac z2-\sec^2\frac z2=-1\,.$$
\end{proof}

%{\color{red}Formula \eqref{eq:J} is essentially known: it is the specialization to the polar map of the general expression for admissible complex structures on $G\times\lieg$ given in \cite[Prop.~3.7 and (3.19)]{HuebschmannLeicht}, going back to \cite[Prop.~3.1]{Bielawski:2003}; neither compactness of $G$ nor definiteness of $\ka$ plays any role in the computation.}

\begin{convention}\label{conv:g}
	$\omega:=d\varpi$ on $T^*G$, and $\gad(u,v):=\omega(\Jad u,v)$; the
	sign is fixed so that $\gad$ is positive in the compact/Riemannian
	case (Remark \ref{rem:flatlimit}). In the notation of the previous sections, $\gad=E^*\gtG$ under \eqref{eq:ident}.
\end{convention}

\begin{proposition}\label{prop:blocks}
	At $(q,x)$,
	\begin{align}
		\omega\bigl((\xi_1,\eta_1),(\xi_2,\eta_2)\bigr)
		&=\ka(\xi_1,\ad_x\xi_2)-\ka(\xi_1,\eta_2)+\ka(\eta_1,\xi_2),
		\label{eq:omega}
	\end{align}
	\begin{align}
		\gad\bigl((\xi_1,\eta_1),(\xi_2,\eta_2)\bigr)
		=&\ka\bigl(\fsym(\ad_x)\xi_1,\xi_2\bigr)
		+\ka\bigl(\xi_1,\tsym(\ad_x)\eta_2\bigr) \notag\\
		&+\ka\bigl(\xi_2,\tsym(\ad_x)\eta_1\bigr)
		+\ka\bigl(\vsym(\ad_x)\eta_1,\eta_2\bigr).
		\label{eq:g}
	\end{align}
	The triple $(\omega,\Jad,\gad)$ is pseudo-K\"ahler on $T^*G$:
	$$\omega(\Jad\cdot,\Jad\cdot)=\omega,\quad
	\gad(\Jad\cdot,\Jad\cdot)=\gad,\quad \omega=\gad(\cdot,\Jad\cdot),\quad
	d\omega=0$$ and $\Jad$ integrable.
\end{proposition}

\begin{proof}
	First we prove equation \eqref{eq:omega}. Since $d\varpi$ is tensorial, we may compute it on any extensions of the given tangent vectors; we extend $(\xi_i,\eta_i)$ to the vector field $U_i$ that is constant in the trivialization \eqref{eq:ident}, i.e.\ left-invariant on the $G$-factor, $U_i(q,x)=\bigl(\xi_i^{L}(q),\eta_i\bigr)$. Then $\varpi_{(q,\nu)}(U_i)=\pair{\nu}{\xi_i}$ depends on $\nu$ alone and linearly, while $[U_1,U_2]=\bigl([\xi_1,\xi_2]^{L},0\bigr)$, and the Cartan formula gives
	$$d\varpi(U_1,U_2)=\pair{\eta_1^\flat}{\xi_2}-\pair{\eta_2^\flat}{\xi_1}
	-\pair{\nu}{[\xi_1,\xi_2]},$$ and
	$\ka(x,[\xi_1,\xi_2])=\ka(\ad_x\xi_1,\xi_2)$ by invariance. For
	\eqref{eq:g}, insert \eqref{eq:J} into \eqref{eq:omega}: e.g.\ on
	horizontal pairs,
	\begin{align*}
	\gad\bigl((\xi,0),(\eta,0)\bigr)
	&=\omega\Bigl(\bigl(\tsym(\ad_x)\,\xi,\ \tfrac{1}{\esym(\ad_x)}\,\xi\bigr),\ (\eta,0)\Bigr)
	\\ &=\ka\bigl(\tsym(\ad_x)\,\xi,\ \ad_x\eta\bigr)
	+\ka\bigl(\tfrac{1}{\esym(\ad_x)}\,\xi,\ \eta\bigr)
	\\ &=\ka\Bigl(\bigl(\tfrac{1}{\esym(\ad_x)}-\ad_x\,\tsym(\ad_x)\bigr)\xi,\ \eta\Bigr)
	=\ka\bigl(\fsym(\ad_x)\,\xi,\ \eta\bigr),
	\end{align*}
	and $\frac{z}{\sin z}-z\tan\frac z2=z\cot z=\fsym(z)$ by
	$\cot z+\tan\frac z2=\csc z$. The remaining blocks are obtained in the same way, the only tool being \eqref{eq:symbolrules} for the odd function $F(z)=z$, i.e.\ $\ka(\zeta,\ad_x\zeta')=-\ka(\ad_x\zeta,\zeta')$. On $(\xi_1,\eta_2)$ one gets $-\ka(\tsym(\ad_x)\xi_1,\eta_2)=\ka(\xi_1,\tsym(\ad_x)\eta_2)$, since $\tsym$ is odd; on $(\eta_1,\xi_2)$,
	\begin{align*}
	-\ka\bigl(\vsym(\ad_x)\eta_1,\ad_x\xi_2\bigr)-\ka\bigl(\tsym(\ad_x)\eta_1,\xi_2\bigr)
	&=\ka\bigl(\bigl(\ad_x\vsym(\ad_x)-\tsym(\ad_x)\bigr)\eta_1,\ \xi_2\bigr) \\
	&=\ka\bigl(\xi_2,\tsym(\ad_x)\eta_1\bigr),
	\end{align*}
	because $z\,\vsym(z)-\tsym(z)=\tsym(z)$; and on $(\eta_1,\eta_2)$ the single term $\ka(\vsym(\ad_x)\eta_1,\eta_2)$ appears directly. Closedness of $\omega=d\varpi$ is automatic, and $\Jad$ is integrable
	because it is pulled back from $G_\C$.
\end{proof}

\begin{lemma}\label{lem:invariance}
	The lifted left and right actions on $T^*G$ read
	$L_g(q,x)=(gq,x)$ and $R_h(q,x)=(qh^{-1},\Ad_hx)$ in \eqref{eq:ident}; both preserve $\varpi$, $\omega$,
	$\Jad$ and $\gad$.
\end{lemma}

\begin{proof}
	Cotangent lifts preserve $\varpi$; and
	$P(gq,x)=g\,P(q,x)$,
	$$P(qh^{-1},\Ad_hx)=P(q,x)\,h^{-1},$$ with translations of $\tilde G$
	biholomorphic; then $\gad=\omega(\Jad\cdot,\cdot)$ is preserved.
\end{proof}

\begin{remark}\label{rem:flatlimit}
	At $x=0$, \eqref{eq:g} is $\ka\oplus\ka$ and \eqref{eq:J} is the flat
	$J(\xi,\eta)=(-\eta,\xi)$. If $\ka$ is definite,
	$\Spec(\ad_x)\subseteq i\R$ and $\fsym(it)=t\coth t\geq1$,
	$$\vsym(it)=\frac2t\tanh\frac t2>0.$$ The classical positive adapted
	K\"ahler metric, cf.\ \cite{Szoke,paogal}. The opposite convention
	$\gad=\omega(\cdot,\Jad\cdot)$ flips the global sign of $\gad$ and of
	every Gram form below.
\end{remark}

%\subsection{The Gram form on $N$}\label{sec:gram}

\begin{assumption}\label{ass:H}
	$H\subseteq G$ is a closed connected subgroup satisfying
	\begin{equation}\label{eq:H}
		\lieh\cap\lieh^{\perp_{\ka}}=\{0\} .
	\end{equation}
\end{assumption}

\begin{lemma}\label{lem:H}
	Condition \eqref{eq:H} is equivalent to each of the {following}: $\kah:=\ka|_\lieh$
	non-degenerate; $\ka|_{\lieh^{\perp_\ka}}$ non-degenerate; $\lieg=\lieh\oplus\lieh^{\perp_\ka}$. 
\end{lemma}

\begin{proof}
	Standard linear algebra for the pair $(\lieh,\lieh^{\perp_\ka})$, using
	$$\dim\lieh+\dim\lieh^{\perp_\ka}=d_G.$$
\end{proof}

{For ease of notation, we adopt the following convention $\lies:=\lieh^{\perp_\ka}$. Furthermore, we write $P_\lieh$ for the corresponding $\ka$-orthogonal projection, and $x_\mu\in\lieh$ for the $\kah$-dual of $\mu\in\lieh^*$. Now, we introduce the Hamiltonian $H$-space $Y$ and we study the Gram form of the induced pseudo-Riemannian structure on $N$. }

\begin{assumption}\label{ass:Y} The quadruple
	$(Y,\omega_Y,J_Y,\Psi)$ is a Hamiltonian $H$-space whose symplectic form is pseudo-K\"ahler, $J_Y$ is an integrable complex structure compatible with $\omega_Y$, the metric $g_Y:=\omega_Y(J_Y\cdot,\cdot)$ is non-degenerate \textup(of arbitrary signature\textup), $H$ preserves $\omega_Y$ and
	$J_Y$ \textup(hence $g_Y$\textup), and the equivariant momentum map
	$\Psi:Y\to\lieh^*$ satisfies
	$d\pair{\Psi}{\eta}=\iota_{\eta_Y}\omega_Y$.
\end{assumption}

Recall from the introduction the data \eqref{eq:N}--\eqref{eq:Ind}: $N=T^*G\times Y$ with $\omega_N=d\varpi+\omega_Y$, the $H$-action $h\cdot(p,y)=(ph^{-1},h(y))$ with equivariant momentum map $\psi(p,y)=\Psi(y)-q^{-1}p|_{\lieh}$, and $N/\!\!/H=\psi^{-1}(0)/H$. We endow $N$ with
\begin{equation}\label{eq:JN}
	J_N:=\Jad\oplus J_Y,\qquad g_N:=\gad\oplus g_Y :
\end{equation}
by Proposition \ref{prop:blocks} and Assumption \ref{ass:Y},
$(N,\omega_N,J_N,g_N)$ is a pseudo-K\"ahler manifold, with $H$ acting by pseudo-K\"ahler automorphisms.

\begin{lemma}\label{lem:level}
	Under Assumption \textup{\ref{ass:H}},
	\[
	\psi^{-1}(0)=\bigl\{(p,y):\,q^{-1}p\,\big|_\lieh=\Psi(y)\bigr\}
	\;\cong\;G\times\lies\times Y,
	\]
	a point being $z=\bigl((q,\,x_{\Psi(y)}+\lambda),\,y\bigr)$,
	$\lambda\in\lies$, in the coordinates \eqref{eq:ident}. The
	infinitesimal generator of $\xi\in\lieh$ at $z$ is
	\begin{equation}\label{eq:gen}
		\xi_N(z)=\Bigl(\bigl(-\xi,\ -\ad_x\,\xi\bigr),\ \xi_Y(y)\Bigr),
		\qquad x:=x_{\Psi(y)}+\lambda ,
	\end{equation}
	and the $H$-action on $\psi^{-1}(0)$ is free and proper.
\end{lemma}

\begin{proof}
	{Since $\ka(\lambda,\eta)=0$ for $\lambda\in\lies$ and $\eta\in\lieh$,  $$\ka(x,\eta)=\pair{\Psi(y)}{\eta}$$ for all $\eta\in\lieh$ reads
		$$\kah\bigl(P_\lieh x,\eta\bigr)=\pair{\Psi(y)}{\eta}\,.$$ It determines
		$P_\lieh x=x_{\Psi(y)}$ uniquely, by non-degeneracy of $\kah$, and leaves
		the component $\lambda:=x-P_\lieh x\in\lies$ free \textup(Lemma
		\ref{lem:H}\textup). Hence $x=x_{\Psi(y)}+\lambda$ and
		$\psi^{-1}(0)\cong G\times\lies\times Y$ with coordinates
		$(q,\lambda,y)$.}

	In the coordinates $(q,\nu)$,
	$\nu=q^{-1}p$, the action reads
	$$h\cdot(q,\nu,y)=(qh^{-1},\Ad^*_h\nu,h(y)),$$ whose generator at $z$ is
	$(-\xi,\ \ad^*_\xi\nu,\ \xi_Y(y))$; and
	$(\ad^*_\xi\nu)^\sharp=[\xi,x]=-\ad_x\xi$ by invariance of $\ka$.
	Freeness and properness hold already on the $T^*G$ factor \cite{bou}.
\end{proof}

\begin{theorem}\label{thm:gram}
	At $z=\bigl((q,x),y\bigr)\in\psi^{-1}(0)$ and for $\xi,\eta\in\lieh$,
	\begin{equation}\label{eq:B}
		B_z(\xi,\eta):=g_N\bigl(\xi_N(z),\eta_N(z)\bigr)
		=\ka\bigl(\fsym(\ad_x)\,\xi,\ \eta\bigr)
		\;+\;\Gform_y(\xi,\eta),
	\end{equation}
	where
	\begin{equation}\label{eq:Gcal}
		\Gform_y(\xi,\eta)=g_Y\bigl(\xi_Y(y),\eta_Y(y)\bigr)
		=-\bigl\langle d\Psi_y\bigl(J_Y\,\xi_Y(y)\bigr),\ \eta\bigr\rangle .
	\end{equation}
	For Loeb-type $G$ the operator $\fsym(\ad_x)$ is defined and invertible
	at every level point \textup(Lemma \ref{lem:loeb}\textup); note however
	that for $\lieh\subsetneq\lieg$ the first summand of \eqref{eq:B} is a
	\emph{compression} of $\fsym(\ad_x)$ to $\lieh$, since $\ad_x$ need not
	preserve $\lieh$.
\end{theorem}

\begin{proof}
	{Put $A=\ad_x$. Inserting the $T^*G$-component $v_\xi=(-\xi,-A\xi)$ of
	\eqref{eq:gen} into \eqref{eq:g}:
	\[
	\gad(v_\xi,v_\eta)
	=\ka(\fsym(A)\xi,\eta)+\ka(\xi,\tsym(A)A\eta)+\ka(\eta,\tsym(A)A\xi)
	+\ka(\vsym(A)A\xi,A\eta).
	\]
	The last three terms cancel. Set $\varphi(z):=z\,\tsym(z)=z\tan\frac z2$,
	an even function, so that $\varphi(A)$ is $\ka$-symmetric by
	\eqref{eq:symbolrules}. The third term equals $\ka(\varphi(A)\xi,\eta)$ by
	symmetry of $\ka$; the second equals it by $\ka$-symmetry of $\varphi(A)$; and the fourth, using $z\,\vsym(z)=2\,\tsym(z)$ and
	$\ka(\zeta,A\eta)=-\ka(A\zeta,\eta)$ \textup(that is,
	\eqref{eq:symbolrules} for the odd function $F(z)=z$\textup), equals
	$-2\,\ka(\varphi(A)\xi,\eta)$.}
	
	For \eqref{eq:Gcal}:
	the momentum property gives
	$\pair{d\Psi(v)}{\eta}=\omega_Y(\eta_Y,v)$, whence
	$$-\pair{d\Psi(J_Y\xi_Y)}{\eta}
	=-\omega_Y(\eta_Y,J_Y\xi_Y)
	=\omega_Y(J_Y\xi_Y,\eta_Y)=g_Y(\xi_Y,\eta_Y).$$
\end{proof}

The same argument, applied to the $H$-action on $N$ itself,
packages the whole Gram form as a {momentum Hessian}. For any Hamiltonian action on a pseudo-K\"ahler manifold
	$(M,\omega,J,g=\omega(J\cdot,\cdot))$ with momentum map $\Phi$,
	\[
	g(\xi_M,\eta_M)=-\pair{d\Phi(J\,\xi_M)}{\eta} .
	\]
	In particular
	\begin{equation}\label{eq:package}
		B_z(\xi,\eta)=-\bigl\langle d\psi_z\bigl(J_N\,\xi_N(z)\bigr),\
		\eta\bigr\rangle
		\qquad\text{on }N .
	\end{equation}

\subsection{Equivalent conditions and a counterexample}\label{subsec:suff}

In the present setting the orbit non-degeneracy condition of Theorem \ref{thm:B6} reads
\begin{equation}\label{eq:star}\tag{$\star$}
	B_z\ \text{is non-degenerate for every }z\in\psi^{-1}(0).
\end{equation}
The following is a reformulation of Theorem \ref{thm:B6} in our setting.

\begin{theorem}\label{thm:suff}
	If \eqref{eq:star} holds, then $N/\!\!/H=\psi^{-1}(0)/H$ is a smooth manifold carrying the reduced pseudo-K\"ahler structure $\bigl(\gred,J_{\mathrm{red}},\omega_{\mathrm{red}}=\gred(\cdot,J_{\mathrm{red}}\cdot)\bigr)$ of Theorem \textup{\ref{thm:B6}}.
\end{theorem}

%\begin{proof}	$(N,\omega_N,J_N,g_N)$ is pseudo-K\"ahler, the $H$-action preserves all structures, by Lemma \ref{lem:invariance} and Assumption \ref{ass:Y}. Furthermore $\psi$ is its equivariant momentum map \eqref{eq:RZmom}, and freeness and properness on the level hold by Lemma \ref{lem:level}. Thus \eqref{eq:star} is the only hypothesis to apply Theorem \ref{thm:B6}.\end{proof}

Fix $z\in\psi^{-1}(0)$ and set $V:=T_z(H\cdot z)$; by freeness, $\lieh\cong V$ via $\xi\mapsto\xi_N(z)$, and $\dim V=\dim\lieh$. In the following Proposition we collect some equivalent conditions for \eqref{eq:star}.

\begin{proposition}\label{prop:TFAE}
	The following are equivalent:
	\begin{enumerate}
		\item[\textup{(i)}] $g_N|_V$ is non-degenerate \textup(condition \eqref{eq:star} at $z$\textup);
		\item[\textup{(ii)}] the $\kah$-symmetric operator $\mathfrak B_z\in\End(\lieh)$ defined by $$B_z(\xi,\eta)=\kah(\mathfrak B_z\,\xi,\eta)$$ is invertible;
		\item[\textup{(iii)}] the linear map $\lieh\ni\xi\longmapsto d\psi_z\bigl(J_N\,\xi_N(z)\bigr)\in\lieh^*$ is bijective;
		\item[\textup{(iv)}] $J_N$-transversality: $T_zN=T_z\psi^{-1}(0)\ \oplus\ J_N V$;
		\item[\textup{(v)}] $W:=V+J_NV$ is an $\omega_N$-symplectic \textup(equivalently $g_N$-non-degenerate\textup) $J_N$-invariant subspace of $T_zN$; and in that case $T_z\psi^{-1}(0)=V\oplus\mathcal H_z$, where $\mathcal H_z:=W^{\perp_{g_N}}$ is the $J_N$-invariant horizontal space.
	\end{enumerate}
\end{proposition}

\begin{proof}
	First, two automatic facts. Since the orbit lies inside the level, $$\omega_N(\xi_N,\eta_N)=\pair{d\psi(\eta_N)}{\xi}=0\,,$$ then $\omega_N|_V=0\,.$ Hence $g_N(\xi_N,J_N\eta_N)=\omega_N(J_N\xi_N,J_N\eta_N)=\omega_N(\xi_N,\eta_N)=0$: $V\perp_{g_N}J_NV$.
	
	(i)$\Leftrightarrow$(ii): definition of $\mathfrak B_z$, using non-degeneracy of $\kah$ (Lemma \ref{lem:H}). (i)$\Leftrightarrow$(iii): by \eqref{eq:package} the kernel of the map in (iii) is exactly the radical of $B_z$. (iii)$\Leftrightarrow$(iv): $J_Nv\in T_z\psi^{-1}(0)=\Ker d\psi_z$ iff $d\psi_z(J_Nv)=0$, so injectivity in (iii) says $J_NV\cap T_z\psi^{-1}(0)=0$; and $\dim J_NV=\dim\lieh$ equals the codimension of the level, since by freeness the momentum map is submersive there. (i)$\Leftrightarrow$(v): $V\perp_{g_N}J_NV$ and $g_N|_{J_NV}\cong g_N|_V$ ($J_N$ isometry) give $g_N|_W\cong g_N|_V\oplus g_N|_V$, so the two degenerate together; and on the $J_N$-invariant $W$, $\rad_{\omega_N}(W)=J_N\rad_{g_N}(W)$, so $\omega_N|_W$ and $g_N|_W$ degenerate together as well. When these conditions hold, $T_zN=W\oplus W^{\perp_{g_N}}$, the horizontal $\mathcal H_z$ is $J_N$-invariant with $g_N|_{\mathcal H_z}$ non-degenerate, and $T_z\psi^{-1}(0)=(J_NV)^{\perp_{g_N}}=V\oplus\mathcal H_z$.
\end{proof}

\begin{remark}\label{rem:kahlercase}
	If $g_N$ is definite --- $\ka$ definite and $Y$ genuinely K\"ahler --- every restriction of $g_N$ is non-degenerate and all five conditions hold automatically: this is why the classical K\"ahler quotient carries no such hypothesis. In indefinite signature they are genuine conditions, as the following example shows.
\end{remark}

\begin{example}[failure of \eqref{eq:star}]\label{ex:counter}
	Take everything flat: $G=\R^2$ (abelian, hence of Loeb type: $\ad\equiv0$), $\ka=\mathrm{diag}(1,-1)$, $H=\R e_1$, so that $\kah=1>0$ and \eqref{eq:H} holds; the polar map is the global diffeomorphism $T^*\R^2\cong\C^2$ and the adapted structure is the standard flat one. Take
	\[
	Y=\R^2,\qquad \omega_Y=dx\wedge dy,\qquad J_Y=J_{\mathrm{std}},
	\]
	and
	\[
	g_Y=\omega_Y(J_Y\cdot,\cdot)=-(dx^2+dy^2),
	\]
	a flat pseudo-K\"ahler structure of signature $(0,2)$, with $H$ acting by rotations $t\cdot v=e^{t\alpha J_{\mathrm{std}}}v$, $\alpha>0$: rotations preserve $\omega_Y$ and $J_Y$, and $\pair{\Psi(v)}{1}=-\tfrac{\alpha}{2}|v|^2$ is an equivariant momentum map, so Assumption \ref{ass:Y} holds. The generator of $\xi=1$ is $\xi_Y(v)=\alpha J_{\mathrm{std}}v$, so
	\[
	\Gform_v(1,1)=g_Y(\xi_Y,\xi_Y)=\alpha^2 g_Y(v,v)=-\alpha^2|v|^2 ,
	\]
	while the $T^*G$ part of \eqref{eq:B} is $\ka(\fsym(0)e_1,e_1)=\ka(e_1,e_1)=1$ at every point of the level, which contains points over every $v\in Y$ : explicitly, $$\psi^{-1}(0)=\bigl\{(q,x,v):x_{1}=-\tfrac\alpha2|v|^{2}\bigr\}\cong G\times\lies\times Y,$$ with free coordinates $(q,\lambda=x_{2}e_{2},v)$ as in Lemma \ref{lem:level}. Hence
	\[
	B_z\;=\;1-\alpha^2|v|^2 ,
	\]
	which vanishes exactly on the non-empty hypersurface $\{|v|=1/\alpha\}$ of $\psi^{-1}(0)$: condition \eqref{eq:star} \emph{fails}, although $G$ is of Loeb type, $Y$ is a pseudo-K\"ahler Hamiltonian $H$-space, and \eqref{eq:H} holds. With the opposite orientation, i.e. with $J_Y$ equal to $-J_{\mathrm{std}}$, then $Y$ becomes genuinely K\"ahler, $\Gform_v=+\alpha^2|v|^2$ and $B_z=1+\alpha^2|v|^2>0$ everywhere. 
	%In particular, in indefinite signature no purely algebraic condition on $(\lieg,\lieh,\ka)$ can imply \eqref{eq:star}: it must be verified along the level set, through any of the equivalent forms of Proposition \ref{prop:TFAE}.
\end{example}

\begin{example}[$Y$ a point: the case of $T^*(G/H)$]\label{ex:TGH}
	For $Y=\{pt\}$ one has $\Psi=0$ and $\Gform\equiv0$, the level set is $\psi^{-1}(0)=\{(q,\lambda):\lambda\in\lies\}$ in the coordinates \eqref{eq:ident}, and $N/\!\!/H=T^*G/\!\!/H=T^*(G/H)$, the Marsden--Weinstein quotient at $0\in\lieh^*$ \cite[Theorem 2.2.2]{Marsden:2007}; condition \eqref{eq:star} asks that  $$\ka\bigl(\fsym(\ad_\lambda)\,\cdot\,,\cdot\bigr)\big|_{\lieh\times\lieh}$$ be non-degenerate for every $\lambda\in\lies$. This may fail. Take the oscillator group of Section \ref{sec:oscillatorexample} and
	\[
	\lieh=\R\,(H-E),\qquad \kah=\ka_e(H-E,H-E)=-2\neq0 ,
	\]
	a closed one-parameter subgroup for which \eqref{eq:H} holds, with
	$$\lies=\lieh^{\perp_{\ka}}=\{\,a(H+E)+bP+cQ\,\}\,.$$ For $x=a(H+E)+bP+cQ$ the operator $\ad_x$ has minimal polynomial $\lambda(\lambda^{2}+a^{2})$, whence
	\[
	\fsym(\ad_x)=\Id-h(a)\,\ad_x^{2},
	\qquad
	h(a):=\frac{a\coth a-1}{a^{2}}\quad\bigl(h(0)=\tfrac13\bigr),
	\]
	and, since $\ad_x^{2}(H-E)=a\,(bP+cQ)-(b^{2}+c^{2})\,E$,
	\[
	B_z(H-E,\,H-E)
	=\ka_e\bigl(\fsym(\ad_x)(H-E),\,H-E\bigr)
	=h(a)\,\bigl(b^{2}+c^{2}\bigr)-2 .
	\]
	This vanishes on the non-empty hypersurface $\{h(a)(b^{2}+c^{2})=2\}$ of the level set (e.g.\ $a=0$, $b^{2}+c^{2}=6$): condition \eqref{eq:star} \emph{fails}, and $T^*(G/H)$ does not inherit the reduced pseudo-K\"ahler structure globally. Note, for later use, that the failure disappears as soon as $\lieh$ contains the centre with a hyperbolic pairing: compare Example \ref{ex:oscillator}.
\end{example}

Eventually, we give a simple definite-signature criterion whose proof is left to the reader.

\begin{proposition}\label{prop:definite}
	If $\ka$ is definite \textup(then $G$ is automatically of Loeb type and \eqref{eq:H} holds for every $H$\textup) and $g_Y$ is positive semi-definite on the orbit directions \textup(in particular if $Y$ is genuinely K\"ahler\textup), then $B_z\succeq\kah>0$ at every level point, and \eqref{eq:star} holds.
\end{proposition}

%\begin{proof}	For definite $\ka$, every $\ad_x$ is skew-adjoint, so $\Spec(\ad_x)\subseteq i\R$ and $\fsym(\ad_x)-\Id=h(\ad_x)$ with $h(it)=t\coth t-1\geq0$ even: the operator $\fsym(\ad_x)-\Id$ is positive semi-definite, whence $\ka(\fsym(\ad_x)\xi,\xi)\geq\ka(\xi,\xi)$ for every $\xi\in\lieh$, while $\Gform_y\geq0$ by hypothesis.\end{proof}

\section{Application: the coadjoint orbit}\label{sec:application}

We now specialize the results of Section \ref{sec:polar} to 
\[
Y=\orb=\Ad(G)\mu\subseteq\lieg ,
\]
where we fix a bilinear form $\kappa_e$ as in Section \ref{sec:orbits}, and with the KKS form \eqref{eq:kks} and with the canonical invariant pseudo-K\"ahler structure $(J,\gamma)$ of Theorem \ref{thm:orbit}. Here, we always assume the hypotheses \textup{(K1)} and \textup{(K2)} at $\mu$ (note that \textup{(K2)} is implied by the standing Loeb hypothesis on $G$). The orbit is regarded as an $H$-space for a closed connected subgroup
\[
G_\mu\subseteq H\subseteq G
\]
satisfying \eqref{eq:H}. First, we check that $(\orb,\omega)$, together with the structure of Theorem \ref{thm:orbit}, satisfies Assumption \ref{ass:Y}. Then, we study pseudo-K\"ahler induction for $Y=\orb$.

%\subsection{The orbit as a pseudo-K\"ahler Hamiltonian space}
%\subsection{The case of orbits}

\begin{proposition}\label{prop:orbitH}
	For every closed subgroup $H\subseteq G$, the orbit $\orb$, with the KKS form $\omega$ of \eqref{eq:kks}, the restricted action $h\cdot\nu=\Ad_h\nu$, and the map
	\begin{equation}\label{eq:Psiorb}
		\Psi:\orb\longrightarrow\lieh^*,\qquad
		\pair{\Psi(\nu)}{\eta}:=-\,\kappa_e(\nu,\eta)\quad(\eta\in\lieh),
	\end{equation}
	is a Hamiltonian $H$-space; together with $J_Y:=J$ and $g_Y:=\gamma=\omega(J\cdot,\cdot)$ from Theorem \ref{thm:orbit}, it satisfies Assumption \ref{ass:Y}. %The sign in \eqref{eq:Psiorb} is dictated by the convention \eqref{eq:kks}.
\end{proposition}

\begin{proof}
	Invariance of $\omega$, $J$, $\gamma$ under $H\subseteq G$ is Theorem \ref{thm:orbit}. Furthermore, for $\eta\in\lieh$ and a tangent vector $v=[\zeta,\nu]$, Lemma \ref{lem:orbitalg}(1) gives
	\begin{align*}
	\pair{d\Psi_\nu(v)}{\eta}&=-\kappa_e([\zeta,\nu],\eta)
	=\kappa_e(\nu,[\zeta,\eta])
	=-\kappa_e(\nu,[\eta,\zeta]) \\
	&=\omega\bigl([\eta,\nu],[\zeta,\nu]\bigr)
	=\bigl(\iota_{\eta_{\orb}}\omega\bigr)(v),
	\end{align*}
	the fundamental vector field of $\eta$ being $\eta_{\orb}(\nu)=[\eta,\nu]$. The equivariance follows easily: for $a\in H$ one has $\Ad_a\lieh=\lieh$ and $$\pair{\Psi(\Ad_a\nu)}{\eta}=-\kappa_e(\Ad_a\nu,\eta)=-\kappa_e(\nu,\Ad_{a^{-1}}\eta)=\pair{\Ad^*_a\Psi(\nu)}{\eta},$$ by $\Ad$-invariance of $\kappa_e$. Finally $g_Y=\gamma$ is non-degenerate and $J$ integrable by Theorem \ref{thm:orbit}.
\end{proof}

Thus the specialization makes sense for \emph{every} closed $H\subseteq G$; the requirements are \textup{(K1)}--\textup{(K2)}, which produce the canonical $(J,\gamma)$. The inclusion $G_\mu\subseteq H$ is what makes the following elementary observation available.

\begin{lemma}\label{lem:role}
	Assume $G_\mu\subseteq H$. Then $\mu\in\liegmu\subseteq\lieh$; consequently $[\mu,\lieh]\subseteq\lieh$ and, under \eqref{eq:H}, $\kappa_e(\mu,\lies)=0$ with $\lies=\lieh^{\perp_{\ka}}$. Moreover the $H$-orbit through $\mu$ is $H(\mu)\cong H/G_\mu$, embedded in $\orb\cong G/G_\mu$ with the same stabilizer, and by \cite[Prop.~3.1(a)]{RatiuZiegler} the orbit $\orb$ meets the image of the momentum map of the induced space $\Ind_H^G\orb=N/\!\!/H$.
\end{lemma}

%\begin{remark}[which algebraic hypothesis?]\label{rem:which}	The two conditions play different and complementary roles:	\begin{itemize}		\item \textup{(K1)}: $\liegmu\cap\liegmu^{\perp_{\ka}}=\{0\}$ is a hypothesis on the \emph{fibre}: together with \textup{(K2)} it is exactly what guarantees the existence of the canonical invariant complex structure $J$ on $\orb$ (Theorem \ref{thm:orbit}); it enters nowhere else.		\item $\eqref{eq:H}$: $\lieh\cap\lieh^{\perp_{\ka}}=\{0\}$ is the hypothesis of the \emph{induction framework}: it provides the splitting $\lieg=\lieh\oplus\lies$, the parametrization of the level set (Lemma \ref{lem:level}), and the operator formulation of condition \eqref{eq:star} (Proposition \ref{prop:TFAE}(ii)).	\end{itemize}	Both are therefore imposed here, each for its own purpose; neither implies the other, and --- as Example \ref{ex:counter} shows --- neither, nor any other purely algebraic condition, implies \eqref{eq:star}, which remains the condition to be verified.\end{remark}

%\subsection{The Gram form for the orbit}

\begin{lemma}\label{lem:levelorb}
	With $\Psi$ as in \eqref{eq:Psiorb}, the level set of Lemma \ref{lem:level} is
	\[
	\psi^{-1}(0)=\Bigl\{\bigl((q,\,-\nu+\lambda),\,\nu\bigr):\ q\in G,\ \nu\in\orb,\ \lambda\in\lies\Bigr\}
	\]
	in the coordinates \eqref{eq:ident}. 
	%{\color{red}The group $(\text{left }G)\times H$ acts on it by $g_N$-isometries, the left $G$-factor acting transitively on the $q$-coordinate, and the orbit space of this action is parametrized by} $(\orb\times\lies)/H$, where $h\in H$ acts by $(\nu,\lambda)\mapsto(\Ad_h\nu,\Ad_h\lambda)$.
	The generator of $\xi\in\lieh$ at $z=\bigl((e,x),\nu\bigr)$, $x=-\nu+\lambda$, is $\xi_N(z)=\bigl((-\xi,-\ad_x\xi),\,[\xi,\nu]\bigr)$.
\end{lemma}

\begin{proof}
	By Lemma \ref{lem:H} the $\kah$-dual of $\Psi(\nu)$ is $x_{\Psi(\nu)}=-P_\lieh\nu$, so the fibre over $\nu$ in Lemma \ref{lem:level} is $-P_\lieh\nu+\lies=-\nu+\lies$. %Isometry of the two actions is Lemma \ref{lem:invariance} on the first factor and $G$-invariance of $(J,\gamma)$ (Theorem \ref{thm:orbit}) on the second; equivariance $\Ad_h(-\nu+\lambda)=-\Ad_h\nu+\Ad_h\lambda$ and $\Ad_h\lies=\lies$ give the stated action on $(\nu,\lambda)$. 
	The generator is \eqref{eq:gen} with $\xi_Y(\nu)=[\xi,\nu]$.
\end{proof}

\begin{theorem}\label{thm:gramorb}
	At $z=\bigl((e,-\nu+\lambda),\nu\bigr)$ and for $\xi,\eta\in\lieh$,
	\begin{equation}\label{eq:Borb}
		B_z(\xi,\eta)
		=\kappa_e\bigl(\fsym(\ad_{\nu-\lambda})\,\xi,\ \eta\bigr)
		\;-\;\kappa_e\bigl(J_\nu\,\ad_\nu\,\xi,\ \eta\bigr),
	\end{equation}
	where $J_\nu\,\ad_\nu\,\xi$ is defined because $\ad_\nu\xi\in T_\nu\orb=\im\ad_\nu$, on which $J_\nu$ acts.
\end{theorem}

\begin{proof}
	The first summand is Theorem \ref{thm:gram}, together with the evenness of $\fsym$: $\fsym(\ad_{-\nu+\lambda})=\fsym(-\ad_{\nu-\lambda})=\fsym(\ad_{\nu-\lambda})$. For the second: given $\xi$, pick $\xi'\in\lieg$ with $\ad_\nu\xi'=J_\nu\ad_\nu\xi$ --- possible since $J_\nu\ad_\nu\xi\in\im\ad_\nu$ --- so that $J_\nu[\xi,\nu]=J_\nu(-\ad_\nu\xi)=-\ad_\nu\xi'=[\xi',\nu]$. Then, by \eqref{eq:kks}, Lemma \ref{lem:orbitalg}(1) and $\kappa_e$-antiadjointness of $\ad_\nu$,
	\begin{align*}
	\Gform_\nu(\xi,\eta)&=\gamma\bigl([\xi,\nu],[\eta,\nu]\bigr)
	=\omega\bigl([\xi',\nu],[\eta,\nu]\bigr)
	=-\kappa_e(\nu,[\xi',\eta])
	\\&=\kappa_e(\xi',\ad_\nu\eta)
	=-\kappa_e(\ad_\nu\xi',\eta)
	=-\kappa_e(J_\nu\ad_\nu\xi,\eta). 
	\end{align*}
\end{proof}

%Note the structural feature already visible in Theorem \ref{thm:gram}: for $\lieh\subsetneq\lieg$ the first summand of \eqref{eq:Borb} is a compression to $\lieh$, and the two summands involve $\ad_{\nu-\lambda}$ and $\ad_\nu$ respectively; they combine into a single operator exactly at the \emph{aligned} points $\lambda=0$. There, the following elementary identity becomes decisive.

{
	\begin{remark}\label{rem:alignedexp}
		Assume $G_\mu\subseteq H$. By Lemma \ref{lem:levelorb} the free parameter of
		the level set is $\lambda\in\lies$, and the two summands of \eqref{eq:Borb}
		involve $\ad_{\nu-\lambda}$ and $\ad_\nu$ respectively: for $\lambda=0$ they
		combine into a single operator. Moreover, for $\nu\in H(\mu)$ one has
		$\nu\in\lieh$, so $\ad_\nu$ preserves $\lieh$ and it maps $\lieh$
		into $W:=\lieh\cap T_\nu\orb$, which is $\ad_\nu$-invariant, hence
		$J_\nu$-invariant, $J_\nu=f(A'_\nu)$ being a polynomial in $A'_\nu$. Since
		$J_\nu$ commutes with $\ad_\nu$ on $W$ and $J_\nu^2=-\Id_W$, on $\lieh$
		\[
		\fsym(\ad_\nu)-J_\nu\!\circ\!\ad_\nu
		=\esym(\ad_\nu)^{-1}\bigl(\cos\ad_\nu-J_\nu\sin\ad_\nu\bigr),
		\]
		a product of invertible operators, the second having inverse
		$$\cos\ad_\nu+J_\nu\sin\ad_\nu\,.$$ Hence $B_z$ is non-degenerate at every level
		point with $\lambda=0$ and $\nu\in H(\mu)$. 
		%and, by continuity of	$\det\mathfrak B_z$, on an open neighbourhood of them; its global validity remains, in general, a condition on the family $(\nu,\lambda)\in(\orb\times\lies)/H$, to be checked through Proposition		\ref{prop:TFAE}.
	\end{remark}
}

\subsection{The oscillator group revisited}

We conclude by carrying out the whole program on the example of Section \ref{sec:oscillatorexample}, whose notation we retain: $\lieg=\operatorname{span}\{H,P,Q,E\}$ with $[H,P]=Q$, $[H,Q]=-P$, $[P,Q]=E$, $E$ central, and $\kappa_e$ determined by $\{P,Q\}$ orthonormal, $\kappa_e(E,H)=1$, $\kappa_e(H,H)=\kappa_e(E,E)=0$ (signature $(3,1)$). This group is of Loeb type: for $x=aH+bP+cQ+dE$ one computes $\Spec(\ad_x)=\{0,0,\pm ia\}\subseteq i\R$.

\begin{example}\label{ex:oscillator}
	Let $\mu:=H+E\in\lieg$, the $\kappa_e$-dual of the covector $H'+E'$ of Section \ref{sec:oscillatorexample}, so that $\orb$ is the paraboloid orbit
	\[
	\orb=\Bigl\{\nu=H+h_PP+h_QQ+\bigl(1+\tfrac{h_P^2+h_Q^2}{2}\bigr)E:\ (h_P,h_Q)\in\R^2\Bigr\},
	\]
	with $\ad_\mu=\ad_H$ semisimple of spectrum $\{0,0,\pm i\}$: \textup{(K1)} and \textup{(K2)} hold, $\liegmu=\operatorname{span}\{H,E\}$, $\liegmu^{\perp_{\ka}}=\operatorname{span}\{P,Q\}$, and $\kappa_e|_{\liegmu^{\perp_{\ka}}}>0$, so by Theorem \ref{thm:orbit} the canonical structure $(J,\gamma)$ is \emph{K\"ahler}: it is the flat K\"ahler structure of the harmonic oscillator found in Section \ref{sec:oscillatorexample}. Explicitly, a direct computation with $J_\nu=f(A'_\nu)=A'_\nu$ \textup(here $-A'^2_\nu=\Id$\textup) gives
	\[
	\Gform_\nu(H,H)=-\kappa_e(J_\nu\ad_\nu H,H)=h_P^2+h_Q^2\geq0,
	\qquad \Gform_\nu(E,\cdot)=0 ,
	\]
	consistently with $[E,\cdot]=0$.
	
	As intermediate subgroup take $H:=G_\mu=\exp\bigl(\operatorname{span}\{H,E\}\bigr)$, abelian and closed --- the \emph{only} proper choice with $G_\mu\subseteq H\subsetneq G$, since no $3$-dimensional subalgebra of $\lieg$ contains $\liegmu$: for $0\neq v\in\operatorname{span}\{P,Q\}$ the bracket $[H,v]$ is a nonzero vector of $\operatorname{span}\{P,Q\}$ orthogonal to $v$, hence outside $\liegmu\oplus\R v$. Then
	\[
	\kah=\begin{pmatrix}0&1\\1&0\end{pmatrix}\ \text{ on }(H,E),
	\qquad
	\lieh\cap\lieh^{\perp_{\ka}}=\{0\}:
	\]
	hypothesis \eqref{eq:H} holds, with $\kah$ a \emph{hyperbolic} plane. The momentum map \eqref{eq:Psiorb} is, up to sign, the harmonic-oscillator pair: $\pair{\Psi(\nu)}{H}=-\bigl(1+\tfrac{h_P^2+h_Q^2}{2}\bigr)$, $\pair{\Psi(\nu)}{E}=-1$.
	
	At the level point $z=\bigl((e,-\nu+\lambda),\nu\bigr)$, $\lambda=\lambda_1P+\lambda_2Q\in\lies=\operatorname{span}\{P,Q\}$, the Gram form on $\lieh$ in the basis $(H,E)$ is computed from \eqref{eq:Borb}: writing $x=-\nu+\lambda$ \textup(with $H$-coefficient $-1$\textup), one finds from $\ad_x^3H=-\ad_xH$ that
	\[
	\kappa_e\bigl(\fsym(\ad_x)H,H\bigr)=(\coth 1-1)\bigl[(h_P-\lambda_1)^2+(h_Q-\lambda_2)^2\bigr],
	\]
	\[
	\kappa_e\bigl(\fsym(\ad_x)E,\eta\bigr)=\kappa_e(E,\eta),
	\]
	the latter because $\ad_xE=0$, hence $\fsym(\ad_x)E=E$. Therefore
	\begin{equation}\label{eq:oscB}
		B_z=\begin{pmatrix}
			\beta_z & 1\\[1mm] 1 & 0
		\end{pmatrix},
		\qquad
		\beta_z=(\coth 1-1)\bigl[(h_P-\lambda_1)^2+(h_Q-\lambda_2)^2\bigr]
		+\bigl(h_P^2+h_Q^2\bigr)\ \geq0,
	\end{equation}
	and $\det B_z=-1$ for \emph{every} $z$: condition \eqref{eq:star} holds \emph{identically}. By Theorem \ref{thm:suff},
	\[
	N/\!\!/G_\mu=\Ind_{G_\mu}^{G}\orb
	\]
	is a six-dimensional pseudo-K\"ahler manifold. Its signature is $(6,0)$: indeed $\gad$ has the signature $(6,2)$ of $\kappa_e\oplus\kappa_e$, so $g_N$ has signature $(8,2)$; the Gram form \eqref{eq:oscB} has signature $(1,1)$ everywhere, so $g_N$ restricted to $V\oplus J_NV$ has signature $(2,2)$ \textup(Proposition \ref{prop:TFAE}(v)\textup), and the reduced metric --- the restriction of $g_N$ to the horizontal $(V\oplus J_NV)^{\perp_{g_N}}$ --- has signature $(8,2)-(2,2)=(6,0)$: the reduction is Riemannian.
\end{example}

\begin{comment}
\begin{remark}[the mechanism]\label{rem:mechanism}
	The unconditional validity of \eqref{eq:star} in Example \ref{ex:oscillator} rests on a structural fact worth isolating: the central element $E\in\lieh$ satisfies $\ad_xE=0$ for every $x$ \emph{and} $E_{\orb}=0$ on every coadjoint orbit, so its row of the Gram form is the constant $\kappa_e(E,\cdot)|_\lieh$, independently of the point; the hyperbolic pairing $\kappa_e(E,H)=1$ then forces $\det B_z=-1$ regardless of the entry $\beta_z$. This is a mechanism by which \eqref{eq:star} can hold with $\kah$ \emph{indefinite}, complementary to the definite criterion of Proposition \ref{prop:definite} and in contrast with the failure of Example \ref{ex:counter}, where $\lieh$ was one-dimensional and definite. Consistently with Proposition \ref{prop:aligned}: here $H(\mu)=\{\mu\}$ is trivially $J$-complex, and at the aligned point over $\mu$ one has $\ad_\mu|_\lieh=0$ and $B=\kah$ --- the hyperbolic plane --- which \eqref{eq:oscB} deforms without ever degenerating.
\end{remark}
\end{comment}

{
\subsection{Self-induction reproduces the canonical structure}\label{subsec:reproduce}

The canonical pair $(J,\gamma)$ of Theorem \ref{thm:orbit} was constructed algebraically, through the functional calculus of $A'_\nu$. The induction framework offers a second, global construction: take $H=G$ and take as fibre the orbit itself, $Y=\orb$, so that $\Ind_G^G\orb=(T^*G\times\orb)/\!\!/G$. The next theorem shows that the two constructions agree.

\begin{theorem}\label{thm:reproduce}
	Let $(G,\kappa_e)$ be of Loeb type and let $\orb$ satisfy \textup{(K1)}--\textup{(K2)}, endowed with $(\omega,J,\gamma)$ of Theorem \ref{thm:orbit} and with the momentum map \eqref{eq:Psiorb}. Take $H=G$ --- so that $\lies=0$, and \eqref{eq:H} is the standing non-degeneracy of $\kappa_e$ --- and $Y=\orb$. Then:
	\begin{itemize}
		\item[\textup{(1)}] the level set is $\psi^{-1}(0)=\bigl\{((q,-\nu),\nu):q\in G,\ \nu\in\orb\bigr\}$; every level point has $\lambda=0$, and condition \eqref{eq:star} holds automatically;
		\item[\textup{(2)}] the map $((q,-\nu),\nu)\mapsto\Ad_q\nu$ descends to a diffeomorphism
		\[
		\tilde\pi:\ \Ind_G^G\orb=(T^*G\times\orb)/\!\!/G\ \xrightarrow{\ \sim\ }\ \orb ,
		\]
		which is an isomorphism of pseudo-K\"ahler manifolds:
		\[
		\tilde\pi_{*}\,\omega_{\mathrm{red}}=\omega,\qquad
		\tilde\pi_{*}\,\gred=\gamma,\qquad
		d\tilde\pi\circ J_{\mathrm{red}}=J\circ d\tilde\pi .
		\]
	\end{itemize}
\end{theorem}

\begin{proof}
	For $H=G$ one has $\lies=\lieg^{\perp_{\ka}}=0$, so Lemma \ref{lem:levelorb} gives the stated level set, with $\lambda=0$ everywhere. {Remark \ref{rem:alignedexp} applies with $H=G$: $B_z$ is non-degenerate at every level point,} \eqref{eq:star} holds, and Theorem \ref{thm:suff} applies; the $G$-action on the level is free and proper, being so on the $T^*G$-factor, where it is right translation.

	On the level set, $$h\cdot\bigl((q,-\nu),\nu\bigr)=\bigl((qh^{-1},-\Ad_h\nu),\Ad_h\nu\bigr)$$ and $\Ad_{qh^{-1}}\Ad_h\nu=\Ad_q\nu$. The map descends, onto $\orb$, and it is injective on the quotient. %, since $\Ad_q\nu=\Ad_{q'}\nu'$ means that $h:=q^{-1}q'$ carries one point to the other. 
	As $\dim\Ind_G^G\orb=2\dim(G/G)+\dim\orb=\dim\orb$ and $d\tilde\pi$ is onto (see below), $\tilde\pi$ is a diffeomorphism.

	Every $G$-orbit in the level set meets the slice $\{q=e\}$, so all the structures may be computed at $z=\bigl((e,-\nu),\nu\bigr)$, where $x=-\nu$ and $\ad_x=-\ad_\nu$. Level-tangent vectors at $z$ are the $u=\bigl((\xi,-w),w\bigr)$ with $\xi\in\lieg$, $w\in T_\nu\orb$; the generator of $\zeta\in\lieg$ is $\zeta_N(z)=\bigl((-\zeta,\ad_\nu\zeta),[\zeta,\nu]\bigr)$ (Lemma \ref{lem:levelorb}); and
	\[
	d\tilde\pi\bigl((\xi,-w),w\bigr)=[\xi,\nu]+w ,
	\]
	which is onto $T_\nu\orb$. In particular the \emph{slice lift} $s(w):=\bigl((0,-w),w\bigr)$ satisfies $d\tilde\pi\,s(w)=w$.

	By Theorem \ref{thm:B6}, $\omega_{\mathrm{red}}$ evaluates on the images of \emph{any} level-tangent vectors as $\omega_N$ does on the vectors themselves; through $\tilde\pi$ and the slice lifts,
	\begin{align*}
	(\tilde\pi_{*}\omega_{\mathrm{red}})(w,w')
	&=\omega_N\bigl(s(w),s(w')\bigr) \\
	&=\omega\bigl((0,-w),(0,-w')\bigr)+\omega(w,w')\\
	&=\omega(w,w') ,
	\end{align*}
	the $T^*G$-term vanishing because every summand of \eqref{eq:omega} contains a $\xi$-entry.

	Now, we aim to prove $\tilde\pi_{*}\gred=\gamma$. Then, write $D:=\ad_\nu$ and let $$\mathfrak B:=\fsym(D)-J_\nu\circ D$$ be the operator of {Remark \ref{rem:alignedexp}}, so that $B_z(\zeta,\xi)=\kappa_e(\mathfrak B\zeta,\xi)$: it is invertible, and it preserves the $\kappa_e$-orthogonal splitting $\lieg=\lieg_\nu\oplus T_\nu\orb$ of \textup{(K1)}, being the identity on $\lieg_\nu$. On $T_\nu\orb$ the operators $D$, $J:=J_\nu$ and all the symbols $F(D)$ commute ($J=f(A'_\nu)$ is itself a function of $D|_{T_\nu\orb}$) with $\kappa_e$-adjoints $F(D)^{\ast}=F(-D)$ by \eqref{eq:symbolrules} and $J^{\ast}=-J$, the function $f$ being odd; in particular $\mathfrak B^{\ast}=\mathfrak B$.

	The horizontal component of the slice lift is $h(w)=s(w)-\zeta(w)_N$, where $\zeta(w)\in\lieg$ is determined by $g_N\bigl(h(w),\xi_N\bigr)=0$ for all $\xi\in\lieg$, i.e.\ by
	\[
	\kappa_e\bigl(\mathfrak B\,\zeta(w),\xi\bigr)=B_z\bigl(\zeta(w),\xi\bigr)=g_N\bigl(s(w),\xi_N\bigr) .
	\]
	We compute the right-hand side. By \eqref{eq:g}, with $\tsym(\ad_x)=-\tsym(D)$ and $\vsym(\ad_x)=\vsym(D)$,
	\[
	\gad\bigl((0,-w),(-\xi,\ad_\nu\xi)\bigr)
	=-\kappa_e\bigl(\tsym(D)w,\xi\bigr)+\kappa_e\bigl(D\,\vsym(D)w,\xi\bigr)
	=\kappa_e\bigl(\tsym(D)w,\xi\bigr) ,
	\]
	using $\kappa_e$-antiadjointness of $D$ and the identity $z\vsym(z)=2\tsym(z)$. For the fibre part, the momentum property of \eqref{eq:Psiorb} reads $$\omega\bigl([\xi,\nu],v\bigr)=\pair{d\Psi_\nu(v)}{\xi}=-\kappa_e(v,\xi),$$ i.e.\ $\omega\bigl(v,[\xi,\nu]\bigr)=\kappa_e(v,\xi)$ for every $v\in T_\nu\orb$; hence
	\[
	\gamma\bigl(w,[\xi,\nu]\bigr)=\omega\bigl(Jw,[\xi,\nu]\bigr)=\kappa_e(Jw,\xi) .
	\]
	Altogether $\kappa_e\bigl(\mathfrak B\zeta(w),\xi\bigr)=\kappa_e\bigl((J+\tsym(D))w,\xi\bigr)$ for all $\xi$, whence
	\[
	\zeta(w)=\mathfrak B^{-1}\bigl(J+\tsym(D)\bigr)w\ \in\ T_\nu\orb .
	\]
	Both $s(w)$ and $\zeta(w)_N$ are tangent to the level, so $h(w)\in T_z\psi^{-1}(0)\cap V^{\perp_{g_N}}=\mathcal H_z$ (Proposition \ref{prop:TFAE}), and $d\tilde\pi\,h(w)=w$. Since $\gred$ is the restriction of $g_N$ to $\mathcal H$ (Theorem \ref{thm:B6}), and since $$g_N\bigl(s(w),\zeta(w')_N\bigr)=B_z\bigl(\zeta(w),\zeta(w')\bigr)=g_N\bigl(\zeta(w)_N,\zeta(w')_N\bigr),$$
	we obtain
	\[
	(\tilde\pi_{*}\gred)(w,w')
	=g_N\bigl(h(w),h(w')\bigr)
	=g_N\bigl(s(w),s(w')\bigr)-B_z\bigl(\zeta(w),\zeta(w')\bigr) .
	\]
	The first summand is $\kappa_e\bigl(\vsym(D)w,w'\bigr)+\gamma(w,w')$, by \eqref{eq:g} again. For the second, the adjoint rules above give $\bigl(J+\tsym\bigr)^{\ast}=-\bigl(J+\tsym\bigr)$ and $(J+\tsym)^2=-\Id+2\tsym J+\tsym^2$ on $T_\nu\orb$, so
	\[
	B_z\bigl(\zeta(w),\zeta(w')\bigr)
	=\kappa_e\bigl((J+\tsym)w,\ \mathfrak B^{-1}(J+\tsym)w'\bigr)
	=\kappa_e\bigl(\mathfrak B^{-1}\bigl(\Id-\tsym^2-2\tsym J\bigr)w,\ w'\bigr) .
	\]
	Therefore
	\[
	(\tilde\pi_{*}\gred)(w,w')-\gamma(w,w')
	=\kappa_e\Bigl(\mathfrak B^{-1}\bigl[\vsym\mathfrak B-(\Id-\tsym^2)+2\tsym J\bigr]w,\ w'\Bigr) ,
	\]
	and the bracket vanishes: $\vsym\mathfrak B=\vsym\fsym-\vsym D\,J=\vsym\fsym-2\tsym J$, so the bracket equals $\vsym\fsym-(\Id-\tsym^2)$, which is zero by 
	\[
	\vsym(z)\,\fsym(z)=2\tan\tfrac z2\,\cot z=1-\tan^2\tfrac z2=1-\tsym(z)^2 .
	\]
	Hence $\tilde\pi_{*}\gred=\gamma$.

	Finally, the reduced triple of Theorem \ref{thm:B6} is compatible: $$\omega_{\mathrm{red}}=\gred(\cdot,J_{\mathrm{red}}\cdot).$$ Pushing forward, $\omega=\gamma\bigl(\cdot,(\tilde\pi_{*}J_{\mathrm{red}})\cdot\bigr)$; since also $\omega=\gamma(\cdot,J\cdot)$ and $\gamma$ is non-degenerate, $\tilde\pi_{*}J_{\mathrm{red}}=J$.
\end{proof}

\section{Induction in stages}\label{sec:stages}

Eventually, we conclude by proving a pseudo-K\"ahler version of the induction-in-stages property. Throughout this section
\[
K\subseteq H\subseteq G
\]
are closed connected subgroups of the Loeb-type group $(G,\kappa_e)$, and we assume that both restrictions
\begin{equation}\label{eq:HK}
	\kah:=\kappa_e|_{\lieh\times\lieh}
	\qquad\text{and}\qquad
	\kappa_{\liek}:=\kappa_e|_{\liek\times\liek}\quad \text{are non-degenerate,}
\end{equation}
i.e.\ that \eqref{eq:H} holds for both subalgebras. The manifold $(Y,\omega_Y,J_Y,g_Y,\Psi)$ is a pseudo-K\"ahler Hamiltonian $K$-space, i.e.\ it satisfies Assumption \ref{ass:Y} with $K$ in place of the acting group. At the level of Hamiltonian spaces, the induction-in-stages property reads,
\begin{equation}\label{eq:stages}
	\Ind_H^G\,\Ind_K^H\,Y\;=\;\Ind_K^G\,Y ,
\end{equation}
we refer to \cite{RatiuZiegler}. We show that, under condition \eqref{eq:star} imposed at each stage, the composite space $\Ind_H^G\,\Ind_K^H\,Y$ carries a reduced pseudo-K\"ahler structure (Theorem \ref{thm:stages}). We then compare it with the structure produced by the direct induction on the common underlying symplectic manifold, and prove by an explicit example that the two refinements differ in general, see Example \ref{ex:stagesdiff}.

\begin{lemma}\label{lem:heredity}
	$(H,\kah)$ is again a Lie group of Loeb type with bi-invariant pseudo-Riemannian metric. Consequently the whole framework of Section \ref{sec:psN} --- the entire Grauert tube of Theorem \ref{thm:Szoke}, the closed formulas for the adapted pair $(\Jad,\gad)$ of $T^*H$, and the Gram-form theory of Section \ref{sec:polar} --- applies to $H$.
\end{lemma}

\begin{proof}
	The form $\kah$ is non-degenerate by hypothesis, and it is $\ad(\lieh)$-invariant, being the restriction to a subalgebra of an $\ad(\lieg)$-invariant form. For the spectral condition, fix $\xi\in\lieh$: the complexification $\lieh_\C\subseteq\lieg_\C$ is invariant under $\ad_\xi$, so the characteristic polynomial of $\ad_\xi|_{\lieh}$ divides that of $\ad_\xi$, whence
	\[
	\Spec\bigl(\ad_\xi|_{\lieh}\bigr)\subseteq\Spec(\ad_\xi)
	\qquad\Longrightarrow\qquad
	\Spec\bigl(\ad_\xi|_{\lieh}\bigr)\cap\R\subseteq\{0\},
	\]
	and $0$ does occur, since $\ad_\xi\xi=0$ with $\xi\in\lieh$. Theorem \ref{thm:Szoke} and Section \ref{sec:polar} then apply to $(H,\kah)$.
\end{proof}

The following consequence of Lemma \ref{lem:heredity} is easy to prove and it is left to the reader.

\begin{proposition}\label{prop:stagefibre}
	Assume condition \eqref{eq:star} for the $K$-action on $T^*H\times Y$. Then
	\[
	Z:=\Ind_K^H\,Y=\bigl(T^*H\times Y\bigr)/\!\!/K
	\]
	is a pseudo-K\"ahler manifold and, endowed with the residual Hamiltonian $H$-action --- induced by the cotangent lift of left translations of $H$ \cite[\S1]{RatiuZiegler} --- and with its equivariant momentum map, it satisfies Assumption \ref{ass:Y} for the group $H$.
\end{proposition}

%\begin{proof}	The pseudo-K\"ahler structure on $Z$ is Theorem \ref{thm:suff}, applied to the group $H$ through Lemma \ref{lem:heredity}. That $Z$, with the residual action and momentum map, is a Hamiltonian $H$-space is the content of the induction construction itself \cite[\S1]{RatiuZiegler}; equivariance is part of that statement. What remains to be checked is that $H$ acts by \emph{pseudo-K\"ahler} automorphisms of the reduced structure. The cotangent lift of left translations preserves $\varpi_H$, $\Jad$ and $\gad$ of $T^*H$ (Lemma \ref{lem:invariance}, applied to $H$), acts trivially on the factor $Y$, and commutes with the right $K$-action; hence it preserves the level set of the $K$-momentum map, permutes the $K$-orbits, and preserves the horizontal distribution, which is defined through $g_{N}$-orthogonality to the orbit directions and to their $J_N$-images. Since the reduced pseudo-K\"ahler structure of Theorem \ref{thm:B6} is characterized by the restriction of $(g_N,J_N)$ to the horizontal distribution, it is preserved as well.\end{proof}

\begin{theorem}[pseudo-K\"ahler induction in stages]\label{thm:stages}
	Assume \eqref{eq:HK}, Assumption \ref{ass:Y} for the $K$-action on $Y$, condition \eqref{eq:star} for the $K$-action on $T^*H\times Y$, and condition \eqref{eq:star} for the $H$-action on $T^*G\times Z$. Then
	\[
	\Ind_H^G\,\Ind_K^H\,Y
	\]
	is a pseudo-K\"ahler manifold, whose underlying symplectic manifold is canonically symplectomorphic to $\Ind_K^G\,Y$ by \eqref{eq:stages}.
\end{theorem}

\begin{proof}
	By Proposition \ref{prop:stagefibre} the pair given by the group $H$ and $Z$ satisfies all the standing hypotheses of Section \ref{sec:polar}; Theorem \ref{thm:suff}, applied to the $H$-action on $T^*G\times Z$, produces the pseudo-K\"ahler structure on the quotient, and the symplectic identification is \eqref{eq:stages}.
\end{proof}

\begin{corollary}[the definite case]\label{cor:definitestages}
	If $\kappa_e$ is definite and $(Y,\omega_Y,J_Y)$ is genuinely K\"ahler, then all the instances of condition \eqref{eq:star} --- for the two stages and for the direct induction --- hold automatically: induction in stages holds unconditionally in the K\"ahler category, and both sides of \eqref{eq:stages} are K\"ahler manifolds.
\end{corollary}

\begin{proof}
	The restriction $\kah$ is definite, so Proposition \ref{prop:definite} applies to the inner stage. Moreover, for definite $\kappa_e$ the adapted metric is the classical positive K\"ahler metric of the Grauert tube \cite{Szoke}, so $g_{N}=\gad\oplus g_Y$ on $T^*H\times Y$ is positive definite, and so is its restriction $g_Z$ to the horizontal distribution: $Z$ is K\"ahler. In particular $g_Z\geq0$ on the $H$-orbit directions, and Proposition \ref{prop:definite} applies to the outer stage; the same argument gives \eqref{eq:star} for the direct induction $\Ind_K^GY$.
\end{proof}

{
When condition \eqref{eq:star} holds \emph{also} for the direct $K$-action on $T^*G\times Y$, both sides of \eqref{eq:stages} carry pseudo-K\"ahler structures on the same underlying symplectic manifold, but the two structures may differ in general. The following example proves both statements in the smallest possible setting, $K=\{e\}$ and $Y=\{pt\}$, where the direct induction is $T^*G$ itself with the adapted pair $(\Jad,\gad)$, and \eqref{eq:star} for the direct induction is vacuous.

\begin{example}[the staged and the direct structures differ]\label{ex:stagesdiff}
	Let $G$ be the oscillator group, $K=\{e\}$, $Y=\{pt\}$. The inner stage is vacuous: $Z=\Ind_{\{e\}}^{H}\{pt\}=T^*H$, with the adapted structure of Lemma \ref{lem:heredity} and the $H$-action by cotangent lifts of left translations, whose generator at $(h,u)\in H\times\lieh\cong T^*H$ is $\zeta_Z(h,u)=(\Ad_{h^{-1}}\zeta,0)$ and whose momentum map is $\pair{\Psi_Z(h,u)}{\zeta}=\kah(\Ad_hu,\zeta)$. The fibre term of Theorem \ref{thm:gram} for the outer stage is therefore
	\[
	\Gform(\zeta,\zeta')
	=\gad^{H}\bigl((\Ad_{h^{-1}}\zeta,0),(\Ad_{h^{-1}}\zeta',0)\bigr)
	=\kah\bigl(\fsym(\ad_u)\Ad_{h^{-1}}\zeta,\ \Ad_{h^{-1}}\zeta'\bigr) ,
	\]
	the $\sigma$-form of Remark \ref{rem:sigma} for the group $H$; for \emph{abelian} $H$ it is simply $\kah(\zeta,\zeta')$, and the outer Gram form at a level point $\bigl((q,x),(h,u)\bigr)$ becomes
	\[
	B'_z(\zeta,\zeta')=\kappa_e\bigl(\fsym(\ad_x)\zeta,\zeta'\bigr)+\kah(\zeta,\zeta'),
	\qquad \zeta,\zeta'\in\lieh .
	\]
	Write $x=aH+bP+cQ+dE$, so that, as in Example \ref{ex:TGH}, $\fsym(\ad_x)=\Id-h(a)\ad_x^2$ with $\ad_x^2H=abP+acQ-(b^2+c^2)E$ and $\ad_xE=0$.

	\smallskip
	\emph{(a) $\lieh=\operatorname{span}\{H,E\}$: the staged structure exists globally, but differs from $(\Jad,\gad)$.}
	In the basis $(H,E)$,
	\[
	B'_z=\underbrace{\begin{pmatrix} h(a)\,(b^2+c^2) & 1\\[1mm] 1 & 0\end{pmatrix}}_{\kappa_e(\fsym(\ad_x)\,\cdot\,,\cdot)|_{\lieh}}
	+\underbrace{\begin{pmatrix}0&1\\[1mm]1&0\end{pmatrix}}_{\Gform=\kah}
	=\begin{pmatrix} h(a)\,(b^2+c^2) & 2\\[1mm] 2 & 0\end{pmatrix},
	\quad \det B'_z=-4 :
	\]
	condition \eqref{eq:star} holds at \emph{every} level point --- by the same mechanism as in Example \ref{ex:oscillator}: the row of the central $E$ is constant and hyperbolically paired --- so Theorem \ref{thm:stages} makes $\Ind_H^G(T^*H)$ a globally defined pseudo-K\"ahler manifold, canonically symplectomorphic to $\bigl(T^*G,\Omega=d\varpi\bigr)$ via the stages identification \eqref{eq:stages}, induced by
	\[
	\bigl((q,x),(h,u)\bigr)\longmapsto\bigl(qh,\ \Ad_{h^{-1}}x\bigr)
	\]
	(one checks that this map kills the generators and matches the reduced symplectic form with $\Omega$). We claim that this symplectomorphism is \emph{not} an isomorphism of pseudo-K\"ahler manifolds, i.e.\ that the transported structure $(J',g')$ differs from $(\Jad,\gad)$.

	It suffices to look at the zero section. At the level point $z_0=\bigl((e,0),(e,0)\bigr)$ everything is flat; the differential of the map above is $(\xi,\eta,\alpha,\beta)\mapsto(\xi+\alpha,\ \eta)$, the level condition is $\beta=P_\lieh\eta$, the generators are $\zeta_{N}(z_0)=(-\zeta,0,\zeta,0)$, $\zeta\in\lieh$, and $g_N$-orthogonality to them reads $P_\lieh\xi=\alpha$. Solving these constraints, the horizontal lifts of the $\lieh$-directions of $T^*G$ are, for $\zeta\in\lieh$,
	\[
	\widehat{(\zeta,0)}=\bigl(\tfrac\zeta2,\,0,\,\tfrac\zeta2,\,0\bigr),
	\qquad
	\widehat{(0,\zeta)}=\bigl(0,\,\zeta,\,0,\,\zeta\bigr) ,
	\]
	whence, for $\zeta,\zeta'\in\lieh$,
	\[
	g'\bigl((\zeta,0),(\zeta',0)\bigr)=\tfrac12\,\kappa_e(\zeta,\zeta'),
	\qquad
	g'\bigl((0,\zeta),(0,\zeta')\bigr)=2\,\kappa_e(\zeta,\zeta'),
	\]
	and
	\[
	J'(\zeta,0)=\bigl(0,\tfrac\zeta2\bigr),
	\quad
	J'(0,\zeta)=(-2\zeta,0),
	\]
	while $g'=\gad$ and $J'=\Jad$ on all the $\lies$-directions, with no mixed terms. At $x=0$, however, $\gad=\kappa_e\oplus\kappa_e$ and $\Jad(\xi,\eta)=(-\eta,\xi)$: along $\lieh$, the staged structure \emph{halves} the metric on the base directions and \emph{doubles} it on the fibre directions --- both pairs being compatible with the same symplectic form $\Omega$. Nor is this a boundary effect of the zero section: at every level point with $h=e$ (to which every point can be moved by the $H$-action, which acts by isometries of both structures) the vector $\bigl((0,E),(0,E)\bigr)$ is tangent to the level and $g_N$-orthogonal to all generators --- both facts because $E$ is central, so $\ad_xE=0=\ad_uE$ --- hence horizontal, and it lifts $(0,E)$; pairing it with any horizontal lift $\bigl((\xi^*,\eta^*),(\alpha^*,\beta^*)\bigr)$ of $(0,H)$ gives, by \eqref{eq:g} and \eqref{eq:symbolrules},
	\[
	g'\bigl((0,H),(0,E)\bigr)=\kappa_e(\eta^*,E)+\kah(\beta^*,E)=2\,\kappa_e(\eta^*,E)=2 ,
	\]
	since $\beta^*=P_\lieh\eta^*$ on the level and $\eta^*=H-\ad_x\alpha^*$ pairs with the central $E$ as $\kappa_e(H,E)=1$; whereas
	\[
	\gad\bigl((0,H),(0,E)\bigr)=\kappa_e\bigl(\vsym(\ad_x)H,E\bigr)=\kappa_e(H,E)=1
	\qquad\text{everywhere.}
	\]

	\smallskip
	\emph{(b) $\lieh=\R\,(H-E)$: the staged structure does not even exist globally.}
	For the one-dimensional subalgebra of Example \ref{ex:TGH} the same formula gives the scalar Gram form
	\begin{align*}
	B'_z&=\kappa_e\bigl(\fsym(\ad_x)(H-E),H-E\bigr)+\kah(H-E,H-E) \\
	&=\bigl[h(a)(b^2+c^2)-2\bigr]-2 \\
	&=h(a)(b^2+c^2)-4 ,
	\end{align*}
	which vanishes on a non-empty hypersurface of the level set \textup(e.g.\ $a=0$, $b^2+c^2=12$\textup): condition \eqref{eq:star} \emph{fails} for the outer stage, although it is vacuous for the direct induction.

	%Through different stagings, the same symplectic manifold $(T^*G,\Omega)$ thus receives: its direct adapted structure $(\Jad,\gad)$; a second, globally defined but \emph{different}, pseudo-K\"ahler refinement, case \textup{(a)}; and a third one, defined only off a hypersurface, case \textup{(b)}. Both condition \eqref{eq:star} and the reduced structure genuinely depend on the staging, and the symplectomorphism \eqref{eq:stages} is not, in general, an isomorphism of pseudo-K\"ahler manifolds. This is in sharp contrast with the self-induction of Theorem \ref{thm:reproduce}, where reduction does reproduce the canonical structure: there the fibre structure is itself built from the infinitesimal geometry of $G$, while here the inner stage inserts the adapted structure of the subgroup $H$, which the outer reduction remembers.
\end{example}
}

\begin{remark}
	Another property involving the induced space is a symplectic version of the Frobenius reciprocity. This result was first established in \cite{GuilleminSternberg} for coadjoint orbits, then it was extended to a set-theoretic bijection for Hamiltonian and prequantum spaces in \cite{RatiuZiegler}, and finally refined to a diffeomorphism of diffeological spaces in \cite{Barbieri:2025}. In the present pseudo-K\"ahler setting, however, Frobenius reciprocity encounters additional obstructions. Besides the fact that the reduced symplectic (or prequantum) space need not arise as a quotient, which could be addressed by diffeological means \cite{Barbieri:2025}, the same phenomena that obstruct induction in stages also obstruct a pseudo-K\"ahler refinement of the reciprocity theorem.
\end{remark}

%\begin{remark}[the trivial inner stage]\label{rem:trivialstage}	For $K=\{e\}$, condition \eqref{eq:star} for the inner stage is vacuous and $Z=\Ind_{\{e\}}^HY=T^*H\times Y$, with the product pseudo-K\"ahler structure and the $H$-action given by the cotangent lift of left translations on the first factor. For the outer stage, the fibre term of Theorem \ref{thm:gram} at a level point $\bigl((q,x),\,(h,u,y)\bigr)$ is, for $\xi,\eta\in\lieh$,	\begin{align*}	\Gform(\xi,\eta)	&=\gad\bigl((\Ad_{h^{-1}}\xi,0),(\Ad_{h^{-1}}\eta,0)\bigr) \\	&=\kah\bigl(\fsym(\ad_u)\Ad_{h^{-1}}\xi,\ \Ad_{h^{-1}}\eta\bigr)	=\sigma_{(h,u)}(\xi,\eta),	\end{align*}	the $\sigma$-form of Remark \ref{rem:sigma} for the group $H$: the Gram form of the composite induction is a sum of two $\fsym$-compressions, one for each stage. Consistently, \eqref{eq:stages} gives $\Ind_H^G\bigl(T^*H\times Y\bigr)=\Ind_{\{e\}}^G\,Y=T^*G\times Y$.\end{remark}

\bigskip
\textbf{Tool and computational resource disclosure.}
In accordance with the Leiden Declaration on Artificial Intelligence and Mathematics~\cite{LeidenDeclaration}, we disclose the following. The research questions, the conceptual framework and the mathematical direction of this work are entirely the authors' own. An AI assistant (Anthropic's Claude, July and August 2026) was used interactively in exploratory discussions concerning some of the arguments, in drafting parts of the text and in performing numerical consistency checks on explicit examples. No proof assistant or formal verification system was used. All statements and proofs have been checked in detail by the authors, who take sole and full responsibility for the correctness and adequacy of the arguments and results, and for the accuracy and completeness of the attribution of prior work.

\bigskip
\textbf{Acknowledgments.} Both authors are grateful to S. Gallivanone for discussion about Grauert tubes. The second author is grateful to R.~Sz\H{o}ke for kindly sharing a copy of his paper on Grauert tubes.
	
The second author is a member of GNSAGA (Gruppo Nazionale per le Strutture Algebriche, Geometriche e le loro Applicazioni) of INdAM (Istituto Nazionale di Alta Matematica ``Francesco Severi'') and gratefully acknowledges its support.

\end{document}